\documentclass{amsart}
\usepackage{amssymb}
\usepackage{amsthm}
\usepackage{amsmath}
\usepackage{amsfonts}
\usepackage{latexsym}
\usepackage{amsrefs}
\usepackage{microtype}
\usepackage{color}
\usepackage{mathrsfs}

\usepackage{cleveref}
\usepackage{upref}

\theoremstyle{plain}
\newtheorem{theorem}{\bf Theorem}

\newtheorem{corollary}{\bf Corollary}[section]

\newtheorem{lemma}{\bf Lemma}

\newtheorem{proposition}{\bf Proposition}[section]

\newcommand{\R}{{\mathbb R}}
\newcommand{\rn}{{\mathbb R}^n}

\newcommand{\eps}{\epsilon}

\begin{document}

\title[Blow-up Analysis for Hardy-Sobolev equations]{Hardy-Sobolev Equations with asymptotically vanishing singularity: Blow-up analysis for the minimal energy}
\author{Saikat Mazumdar}
\address{Department of Mathematics, The University of British Columbia, 1984 Mathematics Road, Vancouver, BC, Canada V6T 1Z2}
\email{saikat@math.ubc.ca}

\begin{abstract}
We study the asymptotic behavior of a sequence of positive solutions $(u_{\epsilon})_{\epsilon >0}$  as $\epsilon \to 0$ to the family of equations
\begin{equation*}
\left\{\begin{array}{ll}
\Delta u_{\epsilon}+a(x)u_{\epsilon}= \frac{u_{\epsilon}^{2^*(s_{\epsilon})-1}}{|x|^{s_{\epsilon}}}& \hbox{ in }\Omega\\
u_{\epsilon}=0 & \hbox{ on }\partial\Omega.
\end{array}\right.
\end{equation*}
where   $(s_{\epsilon})_{\epsilon >0}$ is a sequence of positive real numbers such that  $\lim \limits_{\epsilon \rightarrow 0} s_{\epsilon}=0$, $2^{*}(s_{\epsilon}):= \frac{2(n-s_{\epsilon})}{n-2}$ and $\Omega \subset \R^{n}$ is a bounded smooth domain such that $0 \in \partial \Omega$.  When the sequence  $(u_{\epsilon})_{\epsilon >0}$ is uniformly bounded in $L^{\infty}$, then upto a subsequence it  converges strongly to a minimizing solution of the stationary  Schr\"{o}dinger equation with critical growth.  In case the sequence blows up,  we  obtain strong pointwise control  on the blow up sequence,  and then using the Pohozaev identity localize the point of singularity, which in this case can at most be one,  and derive  precise  blow up rates. In particular when $n=3$ or $a\equiv 0$ then blow up can  occur  only at an interior point of  $\Omega$ or the point $0 \in \partial \Omega$. 
\end{abstract}
\date{February 8th, 2017}
\thanks{This work is part of the PhD thesis of the author, funded by "F\'ed\'eration Charles Hermite" (FR3198 du CNRS) and "R\'egion Lorraine". The author acknowledges these two institutions for their supports.}

\maketitle

\section{Introduction}
Let  $\Omega$ be  a  bounded smooth oriented  domain of $\R^{n}$,  $n \geq 3$, such that $0 \in \partial \Omega$. We define the Sobolev space $H^{2}_{1,0}(\Omega)$ as the completion of the space $C^{\infty}_{c}(\Omega)$, the space of compactly supported  smooth functions in $\Omega$, with respect to  the norm $u\mapsto \| u\|_{H^{2}_{1,0}(\Omega)} = \vert \nabla u\Vert_{L^2(\Omega)}$. We let $2^{*}:=\frac{2n}{n-2}$ be the critical Sobolev exponent for the embeding $H^{2}_{1,0}(\Omega)\hookrightarrow L^{p}(\Omega)$. Namely, the embedding is defined and continuous for $1\leq p\leq 2^*$, and it is compact iff $1\leq p<2^*$. Let $a \in C^{1}(\overline{\Omega})$ be such that the operator  $\Delta + a$ is coercive in $\Omega$, that is there exists  $A_{0}>0$ such that $\int_\Omega ( \left| \nabla \varphi \right|^{2} + a \varphi^{2})~ dx \geq A_{0} \| u\|^{2}_{H^{2}_{1,0}(\Omega)}$ for all $\varphi \in H^{2}_{1,0}(\Omega)$. Solutions $u\in C^2(\overline{\Omega})$ to the problem
\begin{equation*}
\left\{\begin{array}{ll}
\Delta u+a(x)u= u^{2^*-1}& \hbox{ in }\Omega\\
u>0 & \hbox{ in }\Omega\\
u=0 & \hbox{ on }\partial\Omega
\end{array}\right.
\end{equation*}
(often referred to as "Brezis-Nirenberg problem") are critical points of the functional
$$u\mapsto J(u):=\frac{\int_{\Omega} \left(~  |\nabla u|^{2} + a u^{2}~ \right) dx}{\left( ~ \int_{\Omega} {|u|^{2^{*}}} ~ dx \right)^{2/2^{*}} }.$$
Here, $\Delta:=-\hbox{div}(\nabla)=-\sum_i\partial_{ii}$ is the Laplacian with minus sign convention. A natural way to obtain such critical points is to find minimizers to this functional, that is to prove that
\begin{align*}
\mu_{a}(\Omega)= \inf \limits_{u \in H^{2}_{1,0}(\Omega) \backslash \{ 0\}} J(u)
\end{align*}
is achieved. There is a huge and extensive litterature on this problem, starting with the pioneering article of Brezis-Nirenberg \cite{bn} in which the authors completely solved the question of existence of minimizers for $\mu_a(\Omega)$ when $a\equiv $constant and $n\geq 4$ for any domain, and $n=3$ for a ball. Their analysis took inspiration from the contributions of Aubin \cite{aubin} in the resolution of the Yamabe problem. The case when $a$ is arbitrary and $n=3$ was solved by Druet \cite{druet} using blowup analysis. 

\medskip\noindent In \cite{GK}, Ghoussoub-Kang suggested an alternative approach by adding a singularity in the equation as follows. For any $s \in [0,2)$, we define
\begin{align*}
2^{*}(s):= \frac{2(n-s)}{n-2} 
\end{align*}
so that $2^{*}=2^{*}(0)$. Consider the  weak solutions $u\in H^{2}_{1,0}(\Omega) \backslash \{ 0\}$ to the problem
\begin{equation*}
\left\{\begin{array}{ll}
\Delta u+a(x)u= \frac{u^{2^*(s)-1}}{|x|^s}& \hbox{ in }\Omega\\
u\geq 0 & \hbox{ in }\Omega\\
u=0 & \hbox{ on }\partial\Omega.
\end{array}\right.
\end{equation*}
Note here that $0\in\partial\Omega$ is a boundary point. Such solutions can be achieved as minimizers for the problem
\begin{align}{\label{mu with s}}
\mu_{s,a}(\Omega)= \inf \limits_{u \in H^{2}_{1,0}(\Omega) \backslash \{ 0\}} \frac{\int \limits_{\Omega} \left(~  |\nabla u|^{2} + a u^{2}~ \right) dx}{\left( ~ \int \limits_{\Omega} \frac{|u|^{2^{*}(s)}}{|x|^{s}}~ dx \right)^{2/2^{*}(s)} }   \qquad \text{ for }  s\in (0,2)
\end{align}
Consider a sequence of positive real numbers $(s_{\epsilon})_{\epsilon >0}$ such that  $\lim \limits_{\epsilon \rightarrow 0} s_{\epsilon}=0$. We let $(u_{\epsilon})_{\epsilon>0}   \in C^{2} \left( \overline{\Omega} \backslash \{ 0\} \right) \cap C^{1}\left( \overline{\Omega}\right)$ such that  
\begin{eqnarray}{\label{the eqn}}
 \left \{ \begin{array} {lc}
          \Delta u_{\epsilon} + a u_{\epsilon} =  \frac{ u_{\epsilon}^{2^{*}(s_{\epsilon})-1} }{|x|^{s_{\epsilon}}}  \qquad \ \  \text{ in } \Omega,\\
          \qquad  \qquad u_{\epsilon} > 0  \qquad  \qquad \quad  \ \text{ in }   \Omega,\\
            \qquad    \qquad  u_{\epsilon} =0  \qquad  \qquad \quad \   \text{ on }  \partial \Omega.
            \end{array} \right. 
\end{eqnarray}
Moreover, we assume that the $(u_\epsilon)$'s are of minimal energy type in the  sense that 
\begin{align}{\label{min energy condition}}
\frac{\int \limits_{\Omega} \left(~ |\nabla u_{\epsilon}|^{2} + a u_{\epsilon}^{2}~ \right)  dx}{\left( ~ \int \limits_{\Omega} \frac{|u_{\epsilon}|^{2^{*}(s_{\epsilon})}}{|x|^{s}}~ dx \right)^{2/2^{*}(s_{\epsilon})} }   =\mu_{s_{\epsilon},a}(\Omega)+o(1) \leq  \frac{1}{K(n,0)}+ o(1)
\end{align}
as $\epsilon\to 0$, where $K(n,0)>0$ is the best constant in the Sobolev embedding defined in \eqref{best Sobolev  constant}. Indeed, it follows from Ghoussoub-Robert \cites{GRgafa,GRimrn} that such a family $(u_\epsilon)_\epsilon$ exists if the the mean curvature of $\partial \Omega$ at  $0$ is negative.

\medskip\noindent In this paper we are interested   in studying the asymptotic behavior of the sequence $(u_{\epsilon})_{\epsilon >0}$ as $\epsilon \rightarrow 0$. As proved in Proposition \ref{weak lim is a sol}, if the weak limit $u_0$ of $(u_\epsilon)_\epsilon$ in $H_{1,0}^2(\Omega)$ is nontrivial, then the convergence is indeed strong and $u_0$ is a minimizer of $\mu_a(\Omega)$. We completely deal  with the case $u_0\equiv 0$, which is  more delicate, in which blow-up necessarily occurs. In the spirit of the $C^0-$theory of Druet-Hebey-Robert \cite{dhr}, our first result is the following:

\begin{theorem}\label{th:1}
Let  $\Omega$ be  a  bounded smooth oriented  domain of $\R^{n}$,  $n \geq 3$ , such that $0 \in \partial \Omega$, and let  $a \in C^{1}(\overline{\Omega})$ be such that the operator  $\Delta + a$ is coercive in $\Omega$. Let $(s_{\epsilon})_{\epsilon >0} \in (0,2)$ be  a sequence such that $\lim \limits_{\epsilon \to 0 } s_{\epsilon}=0$. Suppose that the sequence  $\left(u_{\epsilon} \right)_{\epsilon>0} \in H^{2}_{1,0}(\Omega)$, where for  each  $\epsilon >0$,  $u_{\epsilon}$  satisfies  $(\ref{the eqn})$ and $(\ref{min energy condition})$, is   a  {\emph{blowup sequence}}, i.e 
\begin{align*}
u_{\epsilon}  \rightharpoonup  0  \qquad \text{weakly in }  H^{2}_{1,0}(\Omega) \qquad  \text{ as } ~\epsilon \rightarrow 0
\end{align*}
Then, there exists  $C>0$  such that  for all $\epsilon >0$
\begin{align*}
 u_{\epsilon}(x) \leq C \left( \frac{\mu_{\epsilon}}{ \mu_{\epsilon}^{2} +|x-x_{\epsilon}|^{2} }\right)^{\frac{n-2}{2}} \qquad \text{ for all } x \in  \Omega
\end{align*}
where $\mu_\epsilon^{-\frac{n-2}{2}}=u_\epsilon(x_\epsilon)=\max \limits_{x\in\Omega}u_\epsilon(x)$.
\end{theorem}
\noindent With this optimal pointwise control, we to obtain more informations on the localization of the blowup point $x_0:=\lim_{\epsilon\to 0}x_\epsilon$ and the blowup parameter $(\mu_\epsilon)_\epsilon$. We let $G: \overline{\Omega} \times \overline{\Omega} \setminus \{ (x,x): x\in \overline{\Omega} \}  \longrightarrow \R$ be  the Green's function of the coercive operator $\Delta + a $ in $\Omega$ with Dirichlet boundary conditions.   
For any  $x \in {\Omega}$ we write  $G_{x}$  as:
\begin{align*}
 G_{x}(y)= \frac{1}{(n-2)\omega_{n-1}|x-y|^{n-2}} + g_{x}(y)\hbox{ for }y\in\Omega\setminus\{x\}
\end{align*}
where $\omega_{n-1}$ is the area of the $(n-1)$- sphere. In dimension $n=3$ or when $a\equiv 0$, one has that $g_{x} \in C^{2}( \overline{\Omega } \setminus \{x\})\cap C^{0, \theta}(\Omega)$ for some $0 < \theta <1$, and $g_x(x)$ is defined for all $x\in \Omega$ and is called the mass of the operator $\Delta+a$.

\begin{theorem}\label{th:2}
Let  $\Omega$ be  a  bounded smooth oriented  domain of $\R^{n}$,  $n \geq 3$ , such that $0 \in \partial \Omega$, and let  $a \in C^{1}(\overline{\Omega})$ be such that the operator  $\Delta + a$ is coercive in $\Omega$. Let $(s_{\epsilon})_{\epsilon >0} \in (0,2)$ be  a sequence such that $\lim \limits_{\epsilon \to 0 } s_{\epsilon}=0$. Suppose that the sequence  $\left(u_{\epsilon} \right)_{\epsilon>0} \in H^{2}_{1,0}(\Omega)$, where for  each  $\epsilon >0$,  $u_{\epsilon}$  satisfies  $(\ref{the eqn})$ and $(\ref{min energy condition})$, is   a  {\emph{blowup sequence}}, i.e 
\begin{align*}
u_{\epsilon}  \rightharpoonup  0  \qquad \text{weakly in }  H^{2}_{1,0}(\Omega) \qquad  \text{ as } ~\epsilon \rightarrow 0
\end{align*}
We let $(\mu_\epsilon)_\epsilon\in (0,+\infty)$ and $(x_\epsilon)_\epsilon\in \Omega$ be such that $\mu_\epsilon^{-\frac{n-2}{2}}=u_\epsilon(x_\epsilon)=\max \limits_{x\in\Omega}u_\epsilon(x)$. We define $x_0:=\lim_{\epsilon\to 0}x_\epsilon$ and we assume that
$$x_0\in \Omega\hbox{ is an interior point.}$$
Then
\begin{align*}
 &\lim_{\epsilon \to 0} \frac{s_{\epsilon}}{  \mu_{\epsilon}^{2}} = 2^{*} K(n,0)^{\frac{2^{*}}{2^{*}-2}} d_{n} ~ a(x_{0})  \qquad \text{ for }~  n \geq 5 \\
 &\lim_{\epsilon \to 0} \frac{s_{\epsilon}}{  \mu_{\epsilon}^{2} \log \left( 1/ \mu_{\epsilon}\right)}  =256\omega_{3} K(4,0)^{2}  ~ a(x_{0}) \qquad \text{ for }~  n =4\\
& \lim \limits_{ \epsilon \to 0 } \frac{s_{\epsilon}}{\mu_{\epsilon}^{n-2}} = - n b_{n}^{2}  K(n,0)^{n/2}  g_{x_{0}}(x_{0}) \qquad \text{ for }  n =3\hbox{ or }a\equiv 0.\\
\end{align*}
where $g_{x_{0}}(x_{0})$ the mass at the point $x_{0} \in \Omega$ for the operator $\Delta + a$,  and  
\begin{equation}\label{def:dn:bn}
 d_{n} = \int_{\R^{n}}  \left(  1+   \frac{|x|^{2}}{n(n-2)}  \right)^{-(n-2)}\,dx ~\hbox{  for } n\geq 5 \; ;\;  \displaystyle{b_{n}= \int _{\R^{n}}  \left( 1+\frac{ |x|^{2}}{n(n-2)} \right)^{-\frac{n+2}{2}}\,dx } 
 \end{equation}
and $\omega_{3}$ is the area of the $3$-sphere.  
\end{theorem}
When $x_0\in\partial\Omega$ is a boundary point, we get similar estimates:
\begin{theorem}\label{th:3}
 Let  $\Omega$ be  a  bounded smooth oriented  domain of $\R^{n}$,  $n \geq 3$ , such that $0 \in \partial \Omega$, and let  $a \in C^{1}(\overline{\Omega})$ be such that the operator  $\Delta + a$ is coercive in $\Omega$. Let $(s_{\epsilon})_{\epsilon >0} \in (0,2)$ be  a sequence such that $\lim \limits_{\epsilon \to 0 } s_{\epsilon}=0$. Suppose that the sequence  $\left(u_{\epsilon} \right)_{\epsilon>0} \in H^{2}_{1,0}(\Omega)$, where for  each  $\epsilon >0$,  $u_{\epsilon}$  satisfies  $(\ref{the eqn})$ and $(\ref{min energy condition})$, is   a  {\emph{blowup sequence}}, i.e 
\begin{align*}
u_{\epsilon}  \rightharpoonup  0  \qquad \text{weakly in }  H^{2}_{1,0}(\Omega) \qquad  \text{ as } ~\epsilon \rightarrow 0
\end{align*}
We let $(\mu_\epsilon)_\epsilon\in (0,+\infty)$ and $(x_\epsilon)_\epsilon\in \Omega$ be such that $\mu_\epsilon^{-\frac{n-2}{2}}=u_\epsilon(x_\epsilon)=\max \limits_{x\in\Omega}u_\epsilon(x)$. Assume that 
$$\lim_{\epsilon\to 0}x_\epsilon=x_0\in\partial\Omega.$$
Then
\begin{enumerate}
\item[(1)] If $n=3$ or $a\equiv 0$, then  as $\epsilon \to 0$
\begin{align*}
 \lim_{\epsilon\to 0} \frac{s_{\epsilon}d(x_\epsilon,\partial\Omega)^{n-2}}{\mu_{\epsilon}^{n-2}}=  \frac{n^{n-1}(n-2)^{n-1} K(n,0)^{{n/2}} \omega_{n-1}  }{2^{n-2}}.
\end{align*}
Moreover, $d(x_\epsilon,\partial\Omega)=(1+o(1))|x_\epsilon|$ as $\epsilon\to 0$. In particular $x_0=0$.
\\
\item[(2)]  If $n=4$. Then  as $\epsilon \to 0$
\begin{align*}
  \frac{s_{\epsilon}}{4} \left(  K(4,0)^{-2} + o(1) \right) - \left( \frac{\mu_{\epsilon}}{d(x_\epsilon,\partial\Omega)}\right)^{2} \left(
  32  \omega_{3} + o(1)  \right)=   \mu^{2}_{\epsilon} \log \left( \frac{d(x_\epsilon,\partial\Omega)}{\mu_{\epsilon}}\right) \left[64\omega_3 a(x_{0})+ o(1) \right] 
\end{align*}
\item[(3)] If $n \geq 5$. Then  as $\epsilon \to 0$
\begin{align*}
   \frac{s_{\epsilon}(n-2)}{2n} \left( K(n,0)^{-n/2}+ o(1) \right) -\left( \frac{\mu_{\epsilon}}{d(x_\epsilon,\partial\Omega)}\right)^{n-2} \left(
   \frac{ n^{n-2}(n-2)^{n} \omega_{n-1} }{ 2^{n-1}}  + o(1)  \right) = \mu_{\epsilon}^2 \left[ d_{n} a(x_{0})+ o(1)\right]
\end{align*}
\end{enumerate}
where $d_n$ is as in \eqref{def:dn:bn}. \end{theorem}
\noindent Theorem \ref{th:3} is a particular case of Theorem \ref{th:bdy:bis} proved in Section \ref{sec:loc:bdy}.

\smallskip\noindent The main difficulty in our analysis is due to the natural singularity at $0\in\partial\Omega$. Indeed, there is a balance between two facts. First, since $s_\epsilon>0$, this singularity exists and has an influence on the analysis, and in particular on the Pohozaev identity (see the statement of Theorem \ref{th:2}). But, second, since $s_\epsilon\to 0$, the singularity should cancel, at least asymptotically. In this perspective, our results are twofolds.\par
\smallskip\noindent Theorem \ref{th:1} asserts that the pointwise control is the same as the control of the classical problem with $s_\epsilon=0$: however, to prove this result, we need to perform a very delicate analysis of the blowup with the perturbation $s_\epsilon>0$, even for the initial steps that are usually standard when $s_\epsilon=0$ (these are Sections \ref{sec:prelim} and \ref{sec:blowup:1}).\par
\smallskip\noindent The influence and the role of $s_\epsilon>0$ is much more striking in Theorems \ref{th:2} and \ref{th:3}. Compared to the case $s_\epsilon=0$, the Pohozaev identity (see Section \ref{sec:loc:int}) enjoys an additional term involving $s_\epsilon$ that is present in the statement of Theorems \ref{th:2} and \ref{th:3}. Heuristically, this is due to the fact that the limiting equation $\Delta u=|x|^{-s}u^{2^*(s)-1}$ is not invariant under the action of the translations when $s>0$.

\medskip\noindent Some classical references for the blow-up analysis of nonlinear critical elliptic pdes are Rey \cite{rey}, Adimurthi- Pacella-Yadava \cite{apy},  Han \cite{han}, Hebey-Vaugon \cite{hv duke} and Khuri-Marques-Schoen \cite{kms}. In Mazumdar \cite{mazumdar.jde} the usefulness of blow-up analysis techniques were illustrated by proving  the existence of solution to critical growth polyharmonic problems on manifolds. The analysis of the 3 dimensional  problem by Druet \cite{druet} and the monograph \cite{dhr} by Druet-Hebey-Robert were important sources of inspiration.

\medskip\noindent This paper is organized as follows. In Section \ref{sec:bdd:domain} we recall   general facts on Hardy-Sobolev inequalities and  prove  few useful general and classical statements. Section \ref{sec:prelim} is  devoted to the proof of convergence to a ground state up to rescaling. In Section \ref{sec:blowup:1}, we perform a delicate blow-up analysis to get a first pointwise control on $u_\epsilon$. The optimal control of Theorem \ref{th:1} is proved in Section \ref{sec:blowup:2}. With the pointwise control of Theorem \ref{th:1}, we are able to estimate the maximum of the $u_\epsilon$'s when the blowup point is in the interior of the domain (Section \ref{sec:loc:int}) or on the boundary (Section \ref{sec:loc:bdy}).

\subsection*{Acknowledgements}

I would like to express my deep gratitude to Professor Fr\'ed\'eric Robert  and Professor Dong Ye, my thesis  supervisors, for their patient guidance, enthusiastic encouragement and useful critiques of this work.

\section{Hardy-Sobolev inequality  and the case of a nonzero weak limit}\label{sec:bdd:domain}
\noindent
The space $\mathscr{D}^{1,2}(\R^{n})$ is defined as the completion of the space $C^{\infty}_{c}(\R^{n})$, the space of compactly supported smooth functions in $\R^{n}$,  with respect to the norm $\| u\|_{\mathscr{D}^{1,2}} = \Vert\nabla u\Vert_{L^2(\rn)}$. 
The embedding $\mathscr{D}^{1,2}(\R^n) \hookrightarrow L^{2^{*}}(\R^n)$ is continuous,  and we denote 
the best constant of this embedding by $K(n,0)$  which can be characterised as
\begin{align}{\label{best Sobolev  constant}}
\frac{1}{K(n,0)} = \inf \limits_{u \in \mathscr{D}^{1,2}(\R^{n})\backslash \{ 0\}} \frac{\int \limits_{\R^{n}} |\nabla u|^{2} ~ dx}{\left( ~ \int \limits_{\R^{n}} |u|^{2^{*}}~ dx \right)^{2/2^{*}} }
\end{align}
Interpolating the Sobolev inequality and the Hardy inequality 
\begin{equation}\label{Hardy inequality}
\int \limits_{\R^{n}} \frac{|u(x)|^{2}}{|x|^{2}}~ dx  \leq \left(  \frac{2}{n-2} \right)^{2}  \int \limits_{\R^{n}} |\nabla u|^{2} ~ dx\hbox{ for }u \in \mathscr{D}^{1,2}(\R^{n}),
\end{equation}
we get the so-called "Hardy-Sobolev inequality" (see \cite{GK} and the references therein): there exists  a constant $K(n,s)>0$ such that 
\begin{align}{\label{best HS constant}}
\frac{1}{K(n,s)} =  \inf \limits_{u \in \mathscr{D}^{1,2}(\R^{n})\backslash \{ 0\}} \frac{\int \limits_{\R^{n}} |\nabla u|^{2} ~ dx}{\left( ~ \int \limits_{\R^{n}} \frac{|u|^{2^{*}(s)}}{|x|^{s}}~ dx \right)^{2/2^{*}(s)} }   
\end{align}
As one checks, $\displaystyle{\lim \limits_{s\to 0} K(n,s)= K(n,0)}$. For a domain  $\Omega\subset \R^n$, we also have:
\begin{proposition}{\label{lim of mu}} $\lim \limits_{s \rightarrow 0} \mu_{s,a}(\Omega) =  \mu_{a}(\Omega)$.
\end{proposition}
\begin{proof} Let $u \in H^{2}_{1,0}(\Omega) \backslash \{ 0 \}$.  H\"{o}lder and Hardy inequalities yield
\begin{align*}
\left(  ~ \int \limits_{\Omega} \frac{|u(x)|^{2^{*}(s)}}{|x|^{s}}~ dx  \right)^{ 2/{2^{*}(s)}} \leq   \left( \frac{2}{n-2}   \right)^{ 2s/{2^{*}(s)}}  \left( ~  \int \limits_{\Omega} |\nabla u(x)|^{2}~ dx  \right)^{{s}/2^{*}(s)}   \left( ~  \int \limits_{\Omega} |u(x)|^{2^{*}} dx  \right)^{\frac{2-s}{2^{*}(s)}}
\end{align*}
 then  the Sobolev inequality gives  that for all $u \in H^{2}_{1,0}(\Omega) \backslash \{ 0 \}$ one has 
\begin{align*}
\frac{\int \limits_{\Omega}  \left( |\nabla u|^{2} + a u^{2} \right)~ dx}{\left( ~ \int \limits_{\Omega} |u|^{2^{*}}~ dx \right)^{2/2^{*}}}  \leq     
  \frac{\int \limits_{\Omega} \left(  |\nabla u|^{2}+ au^{2}  \right)~ dx}{\left( ~ \int \limits_{\Omega} \frac{|u|^{2^{*}(s)}}{|x|^{s}}~ dx \right)^{2/2^{*}(s)} }   \left(   \frac{1}{K(n,0)^{1/2^{*}}} \frac{2}{n-2}   \right)^{s \frac{n-2}{n-s}}   
\end{align*}
So $\mu_{a}(\Omega) \leq  \mu_{s,a}(\Omega)   \left(   \frac{1}{K(n,0)^{1/2^{*}}} \frac{2}{n-2}   \right)^{s \frac{n-2}{n-s}}$.
Passing to limits as $s \to 0$, one obtains that $\mu_{a}(\Omega) \leq \liminf_{s\to 0} \mu_{s,a}(\Omega)$. Let $u \in H^{2}_{1,0}(\Omega) \backslash \{ 0 \}$. By Fatou's lemma one has
\begin{align*}
\int \limits_{\Omega} |u(x)|^{2^{*}}~ dx \leq   \liminf \limits_{s \rightarrow 0}  \int \limits_{\Omega} \frac{|u(x)|^{2^{*}(s)}}{|x|^{s}} ~ dx & \leq   \liminf \limits_{s \rightarrow 0}  \left( \frac{1}{\mu_{s,a}(\Omega)} \int \limits_{\Omega}  \left( |\nabla u|^{2} + a u^{2} \right)~ dx \right)^{2^{*}(s)/2} , \notag\\
\left( ~\int \limits_{\Omega} |u(x)|^{2^{*}}~ dx \right)^{2/2^{*}}& \leq    \liminf \limits_{s \rightarrow 0} \frac{1}{\mu_{s,a}(\Omega)} ~ \int \limits_{\Omega}  \left( |\nabla u|^{2} + a u^{2} \right)~ dx
\end{align*}
Therefore $\limsup_{s \rightarrow 0} \mu_{s,a}(\Omega) \leq \mu_{a}(\Omega)$, hence $\lim \limits_{s \rightarrow 0} \mu_{s,a}(\Omega) = \mu_{a}(\Omega)$. This proves Proposition \ref{lim of mu}.\end{proof}

\noindent The following proposition is standard:
\begin{proposition}{\label{weak lim is a sol}}
 Let  $\Omega$ be  a  bounded smooth oriented  domain of $\R^{n}$,  $n \geq 3$ , such that $0 \in \partial \Omega$. 
Let $a \in C^{1}(\overline{\Omega})$ be such that the operator  $\Delta + a$ is coercive in $\Omega$
Let $(u_{\epsilon})_{\epsilon>0}   \in C^{2} \left( \overline{\Omega} \backslash \{ 0\} \right) \cap C^{1}\left( \overline{\Omega}\right)$ be  as in \eqref{the eqn} and \eqref{min energy condition}. Then there exists $u_0\in H_{1,0}^2(\Omega)$ such that, up to extraction, $u_\epsilon\rightharpoonup u_0$ weakly in $H_{1,0}^2(\Omega)$ as $\epsilon\to 0$. Indeed, $u_{0}  \in C^{2} \left( \overline{\Omega} \backslash \{ 0\} \right) \cap C^{1}\left( \overline{\Omega}\right) $ is a solution to 
\begin{eqnarray*}
 \left \{ \begin{array} {lc}
          \Delta u_{0} + a u_{0} =   u_{0}^{2^{*}-1}  \qquad  & \text{in } \Omega\\
            \qquad  \qquad u_{0} \geq 0  \qquad  \quad \quad  &\ \text{in }   \Omega,\\
           \qquad    \qquad  u_{0} =0  \qquad   \qquad   &\text{on }  \partial \Omega
            \end{array} \right. 
\end{eqnarray*}
If $u_{0} \neq 0$, then $u_0>0$ in $\Omega$ and $\lim_{\epsilon\to 0}u_\epsilon=u_0$ in $C^1(\overline{\Omega})$. Moreover, $\mu_a(\Omega)$ is achieved by $u_0$.
\end{proposition}

\section{Preliminary Blow-up Analysis}\label{sec:prelim}

\medskip

\noindent
We define $\R^{n}_{-}= \{  x \in \R^{n} : x_{1} < 0 \}$ where $x_{1}$  is the first coordinate of a generic point in $\R^{n}$. This space will be the limit space in certain cases after blowup. We describe a  parametrisation around  a point of the boundary $\partial \Omega$. Let $ p  \in \partial \Omega $. Then there exists    $U$,$V$ open in ${\R^{n}}$  and a smooth diffeomorphism $\mathcal{T}: U \longrightarrow  V$ such that upto a rotation of coordinates if necessary 
\begin{equation}\label{def:T:bdy}
\left\{\begin{array}{ll}
\bullet & 0 \in U\hbox{ and }p \in V\\
\bullet & \mathcal{T}(0)=p\\
\bullet &\mathcal{T} \left( U \cap \{x_{1} <0 \} \right)= V \cap   \Omega\\
\bullet &\mathcal{T} \left( U \cap \{x_{1} =0 \} \right)= V \cap \partial  \Omega\\
\bullet & D_{0} \mathcal{T} = \mathbb{I}_{\R^{n}}. \hbox{ Here $D_{x} \mathcal{T} $ denotes the differential of $ \mathcal{T} $ at the point $x$}\\
& \hbox{ and $ \mathbb{I}_{\R^{n}}$ is the identity map on $\R^{n}$.}\\
\bullet &D_0\mathcal{T} (e_{1})= \nu_{p}~\hbox{ where }\nu_{p}\hbox{ denotes the outer unit normal vector to }\\
& \partial \Omega\hbox{ at the  point }p.\\ 
\bullet & \{ D_0\mathcal{T}(e_{2}) , \cdots, D_0\mathcal{T}(e_{n}) \} \hbox{ forms an orthonormal basis of }T_{p} \partial \Omega .
\end{array}\right.
\end{equation}

\medskip\noindent
We start with a   scaling lemma which we shall  employ many times in our analysis.
\begin{lemma}{\label{blowup lemma}}
Let  $\Omega$ be  a  bounded smooth oriented  domain of $\R^{n}$,  $n \geq 3$ , such that $0 \in \partial \Omega$, and let  $a \in C^{1}(\overline{\Omega})$ be such that the operator  $\Delta + a$ is coercive in $\Omega$. Let $(s_{\epsilon})_{\epsilon >0} \in (0,2)$ be  a sequence such that $\lim \limits_{\epsilon \to 0 } s_{\epsilon}=0$. Consider  the sequence  $\left(u_{\epsilon} \right)_{\epsilon>0} \in H^{2}_{1,0}(\Omega)$, where for  each  $\epsilon >0$,  $u_{\epsilon}$  satisfies  $(\ref{the eqn})$ and $(\ref{min energy condition})$.  Let $(y_\epsilon)_\epsilon \in \Omega$, and let  $(\nu_\epsilon)_\epsilon$  and $(\beta_\epsilon)_\epsilon$ be  sequences of positive real numbers defined by 
\begin{align}
\nu_\epsilon^{-\frac{n-2}{2}}= u_{\epsilon}(y_{\epsilon}) \qquad \beta_{\epsilon}:= \left| y_{\epsilon} \right|^{s_{\epsilon}/2} \nu_{\epsilon}^{\frac{2-s_{\epsilon}}{2}} \qquad \text{ for } \epsilon >0
\end{align}
Suppose that $\lim \limits_{\epsilon \to 0} \nu_{\epsilon}=0$ which then implies  that   $\lim \limits_{\epsilon \to 0} \beta_{\epsilon}=0$. Assume that  there exists $C_{1}>0$ such that for any given   $R>0$ one has for  $\epsilon >0$ small
\begin{align}\label{bdds on u 1}
u_{\epsilon}(x) \leq& C_{1} u_{\epsilon}(y_{\epsilon}) ~ \text{ for all } x\in B_{y_{\epsilon}}(R \nu_{\epsilon}) 
\end{align}
Then  $\displaystyle{ \nu_{\epsilon}=  o\left( |y_{\epsilon}| \right)}$ as $\epsilon \to 0$.    Along with the above assumption also suppose that there exists $C_{2}>0$ such that for any given   $R>0$ one has for  $\epsilon >0$ small
\begin{align}\label{bdds on u 2}
u_{\epsilon}(x) \leq &C_{2} u_{\epsilon}(y_{\epsilon}) ~ \text{ for all } x\in B_{y_{\epsilon}} (R \beta_{\epsilon})
\end{align}
Then    $\displaystyle{ \beta_{\epsilon}=  o\left( d(y_{\epsilon}, \partial \Omega) \right)}$  as $\epsilon \to 0$. For   $\epsilon >0$ we then rescale and define 
\begin{align}\label{def:we:lemma}
w_{\epsilon}(x):=  \frac{u_{\epsilon}(y_{\epsilon} + \beta_{\epsilon} x)}{u_{\epsilon}(y_{\epsilon})} \qquad \text{ for } x \in \frac{\Omega-y_{\epsilon}}{\beta_{\epsilon}}
\end{align}
Then there exists $w \in C^{\infty} (\R^{n}) \cap \mathscr{D}^{1,2}(\R^{n})$ such that $ w > 0$ and for  any $\eta \in C^{\infty}_{c}(\R^{n})$
\begin{align*}
\eta w_{\epsilon} \rightharpoonup \eta w \qquad \text{ weakly in } \mathscr{D}^{1,2}(\R^n) \qquad \text{ as } \epsilon \rightarrow 0
\end{align*}
Further, $\lim_{\epsilon\to 0}w_\epsilon=w$ in $C^{1}_{loc}(\R^{n})$ and  $ w$ satisfies  the equation
\begin{eqnarray*}
 \left \{ \begin{array} {lc}
          \Delta w =  w^{2^{*}-1}   \qquad  \text{ in }  \R^{n}\\
          \quad  w \geq 0  \qquad   \qquad    \text{ in }  \R^{n}.\\
          \end{array} \right. 
\end{eqnarray*}
\end{lemma}

\begin{proof}
The proof is completed in the following  steps.

\medskip \noindent {\bf Step 1:} \label{blowup step 1}
We claim that
\begin{align}{\label{first lim}}
\lim \limits_{\epsilon \to 0} \frac{|y_{\epsilon}|}{\nu_{\epsilon}}=+\infty
\end{align}
We prove our claim. Suppose on the contrary that $\frac{|y_{\epsilon}|}{\nu_{\epsilon}}= O(1)$ as $ \epsilon \rightarrow 0$. Then $\lim \limits_{\epsilon \rightarrow + \infty} \left| y_{\epsilon}\right| =0$. Let $\displaystyle{\mathcal{T}: U \to V  }$ be a parametrisation of the boundary as in \eqref{def:T:bdy} at the point $p=0$. For all $\epsilon >0$, we let 
\begin{align*}
\tilde{w}_{\epsilon}(x) = \frac{u_{\epsilon} \circ \mathcal{T}(\nu_{\epsilon  } x)}{ u_{\epsilon}(y_{\epsilon})} \qquad \text{ for } x \in \frac{U }{\nu_{\epsilon}}\cap \{ x_{1} \leq 0 \} 
\end{align*}

\medskip \noindent {\bf Step 1.1:}
 For any $\eta \in C^{\infty}_{c}(\R^{n})$, one has  that  $\eta \tilde{w}_{\epsilon} \in \mathscr{D}^{1,2} (\R^{n}_{-})$ for $\epsilon >0 $ sufficiently small. We claim that there exists $\tilde{w}_{\eta}  \in \mathscr{D}^{1,2}(\R^{n}_{-})$ such that upto a subsequence 
\begin{eqnarray*}
 \left \{ \begin{array} {lc}
          \eta \tilde{w}_{\epsilon} \rightharpoonup  \tilde{w}_{\eta} \qquad  \qquad \text{ weakly in }   \mathscr{D}^{1,2} (\R^{n}_{-}) ~  \text{ as } \epsilon \to 0 \\
          \eta \tilde{w}_{\epsilon}(x) \rightarrow  \tilde{w}_{\eta}(x) \qquad  a.e ~  ~ \text{ in } \R^{n}_{-} ~  \text{ as } \epsilon \to 0 
            \end{array} \right. 
\end{eqnarray*} 

\medskip\noindent
We prove the claim. Let $x \in \R_{-}^{n}$, then  
\begin{align*}
\nabla \left( \eta \tilde{w}_{\epsilon} \right)(x)=   \tilde{w}_{\epsilon}(x) \nabla \eta(x)  +\frac{ \nu_{\epsilon}}{u_{\epsilon}(y_{\epsilon})} ~ \eta(x) D_{(\nu_{\epsilon} x)} \mathcal{T} \left[\nabla u_{\epsilon} \left( \mathcal{T} (\nu_{\epsilon}x) \right) \right]
\end{align*}
\medskip\noindent Now for any $\theta >0$, there exists  $C({\theta}) >0$ such that for any $a,b >0$, $(a+b)^{2} \leq C({\theta}) a^{2} + (1+ \theta) b^{2}$. With this inequality we then obtain 
\begin{align*}
\int \limits_{\R_{-}^{n}} \left|  \nabla \left( \eta  \tilde{w}_{\epsilon} \right)\right|^{2}~ dx \leq C(\theta)  \int \limits_{\R_{-}^{n}} | \nabla \eta |^{2} \tilde{w}_{\epsilon}^{2} ~ dx + (1 + \theta) \frac{\nu_{\epsilon}^{2}}{u^{2}_{\epsilon}(y_{\epsilon})} \int \limits_{\R_{-}^{n}} \eta^{2} \left| 
D_{(\nu_{\epsilon} x)} \mathcal{T} \left[\nabla u_{\epsilon} \left( \mathcal{T}(\nu_{\epsilon}x) \right) \right] \right|^{2} ~ dx
\end{align*}
Since $D_{0} \mathcal{T} = \mathbb{I}_{\R^{n}} $ we have as $\epsilon \to0$ 
\begin{align*}
\int \limits_{\R_{-}^{n}} \left|  \nabla \left( \eta  \tilde{w}_{\epsilon} \right)\right|^{2}~ dx \leq C(\theta)  \int \limits_{\R_{-}^{n}} | \nabla \eta |^{2} \tilde{w}_{\epsilon}^{2} ~ dx + (1 + \theta) \left( 1 + O(\nu_{\epsilon})\right) \frac{\nu_{\epsilon}^{2}}{u^{2}_{\epsilon}(y_{\epsilon})} \int \limits_{\R_{-}^{n}} \eta^{2} \left| \nabla u_{\epsilon} \left( \mathcal{T} (\nu_{\epsilon}x) \right)  \right|^{2} (1+ o(1)) dx 
\end{align*}
With H\"{o}lder inequality and a change of variables  this becomes 
\begin{align}{\label{b'dd on the integral 1}}
\int_{\R_{-}^{n}} \left|  \nabla \left( \eta  \tilde{w}_{\epsilon} \right)\right|^{2}~ dx \leq C(\theta) \left\| \nabla \eta \right\|^{2}_{L^{n}} \left( ~\int_{\Omega} u_{\epsilon}^{2^{*}} ~ dx \right)^{\frac{n-2}{n}}  
+ (1+ \theta)  \left( 1 + O(\nu_{\epsilon})\right)   \int \limits_{\Omega}  \left| \nabla u_{\epsilon} \right|^{2} ~ dx
\end{align}
Since $\left\| u_{\epsilon} \right\|_{H^{2}_{1,0}(\Omega)}= O(1)$ and $\nu_{\epsilon} \rightarrow 0$  as $ \epsilon \rightarrow 0$, so for $\epsilon >0$ small enough, $\left\| \eta \tilde{w}_{\epsilon} \right\|_{\mathscr{D}^{1,2} (\R^{n}_{-})} \leq C_{\eta}$, where $C_{\eta}$ is a constant depending on the function $\eta$. The claim then follows from the reflexivity of $\mathscr{D}^{1,2} (\R^{n}_{-})$.

\medskip \noindent {\bf Step 1.2:} Via a diagonal argument, we get that there exists $\tilde{w} \in   \mathscr{D}^{1,2} (\R^{n}_{-})  $ such  that for any $\eta\in C^\infty_c(\rn)$, then
\begin{eqnarray*}
 \left \{ \begin{array} {lc}
          \eta \tilde{w}_{\epsilon} \rightharpoonup  \eta\tilde{w}\qquad  \qquad \text{ weakly in }   \mathscr{D}^{1,2} (\R^{n}_{-}) ~  \text{ as } \epsilon \to 0 \\
          \eta \tilde{w}_{\epsilon}(x) \rightarrow  \eta\tilde{w}(x) \qquad a.e ~ x~\text{ in } \R^{n}_{-} ~  \text{ as } \epsilon \to 0 
            \end{array} \right. 
\end{eqnarray*}

\noindent We claim that $ \tilde{w} \in  C^{1}( \overline{\R^{n}_{-}})$ and  it satisfies weakly the equation 
\begin{equation}\label{eq:lim:bndy}
 \left \{ \begin{array} {lc}
          \Delta \tilde{w}  =  \tilde{w}^{2^{*}-1}   \qquad  \text{ in }  \R^{n}_{-}\\
        \quad  \tilde{w} =0  \qquad   \qquad \text{on }    \{ x_{1 } =0 \}
            \end{array} \right. 
\end{equation}

\noindent
We prove the claim. For $i,j=1, \ldots,n$, we let $g_{ij}= \left( \partial_{i} \mathcal{T}, \partial_{j} \mathcal{T} \right)$,  the metric induced by the chart $\mathcal{T}$ on the domain $U \cap \{ x_{1} <0\}$ and let $\Delta_{g}$ denote the Laplace-Beltrami operator with respect to the metric $g$. We let 
\begin{align*}
\tilde{g}_{\epsilon}= g\left( \nu _{\epsilon} x\right)
\end{align*}
From  $(\ref{the eqn})$ it follows that for any $\epsilon >0$ and $R>0$, $ \tilde{w}_{\epsilon}$ satisfies  weakly the equation 
\begin{align}{\label{blowup eqn 1}}
\Delta  \tilde{w}_{\epsilon}  + \nu^{2}_{\epsilon} \left( a \circ \mathcal{T}(\nu_{\epsilon}x ) \right)  \tilde{w}_{\epsilon} &=  \frac{    \tilde{w}_{\epsilon} ^{2^{\star}(s_{\epsilon})-1} }{\left| \frac{\mathcal{T}(\nu_{\epsilon} x)}{\nu_{\epsilon}}\right|^{s_{\epsilon}}} & \text{in } B_{0}(R) \cap  \{ x_{1 } <0 \}  \notag\\
       \tilde{w}_{\epsilon}&=0 &  \text{on }  B_{0}(R) \cap  \{ x_{1 } =0 \}
\end{align}
From  \eqref{bdds on u 1} and the properties of the boundary chart $\mathcal{T}$ it follows that  there exists $C_{1}>0$ such that  for $\epsilon >0$ small $0 \leq \tilde{w}_{\epsilon}(x) \leq C_{1}$ for all $x \in B_{0}(R) \cap \{ x_{1} \leq 0 \}$, for $R >0$ large. Then for any $p >1$ there exists a  constant $C_{p}>0$ such that 
\begin{align*}
\int \limits_{B_{0}(R) \cap  \{ x_{1 } <0 \}}  \left[   \frac{  \left(   \tilde{w}_{\epsilon} \right)^{2^{*}(s_{\epsilon})-1} }{\left| \frac{\mathcal{T}(\nu_{\epsilon} x)}{\nu_{\epsilon}}\right|^{s_{\epsilon}}}\right]^{p} ~ dx \leq C_{p}   \int \limits_{B_{0}(R) \cap  \{ x_{1 } <0 \}}  \frac{1}{\left| x\right|^{s_{\epsilon} p}} ~ dx
\end{align*}
So the right hand side of equation $(\ref{blowup eqn 1})$  is uniformly  bounded in   $L^{p}$ for some $p > n$. From standard elliptic estimates  it follows that  the sequence $\left(  \eta_{R} \tilde{w}_{\epsilon} \right)_{\epsilon >0}$ is bounded in $ C^{1,\alpha_{0}}\left(  B_{0}(R)\ \cap \{  x_{1} \leq 0  \} \right)$ for some $\alpha_{0} \in (0,1)$. So by Arzela-Ascoli's theorem and a diagonal argument, we get that  $\tilde{w} \in C^{1, \alpha}_{loc}\left( \rn \cap \{  x_{1} \leq 0  \} \right) $ for  $0<\alpha < \alpha_{0}$, and  that,  up to a subsequence 
\begin{align*}
\lim \limits_{\epsilon \to 0}  \tilde{w}_{\epsilon} = \tilde{w}  \qquad \text{ in } C^{1, \alpha}_{loc}\left(  \rn \cap \{  x_{1} \leq 0  \} \right)
\end{align*}
for  $0<\alpha < \alpha_{0}$. Passing to the limit in \eqref{blowup eqn 1}, we get \eqref{eq:lim:bndy}. This proves our claim.

\medskip \noindent {\bf Step 1.3:}
Let $\tilde{y}_{\epsilon} \in U $ be such that $\mathcal{T} (\tilde{y}_{\epsilon})= y_{\epsilon}$. From the properties \eqref{def:T:bdy} of the boundary chart $\mathcal{T}$, we get that $\frac{\left| \tilde{y}_{\epsilon} \right|}{\nu_{\epsilon}}= O  \left( \frac{\left| {y}_{\epsilon} \right|}{\nu_{\epsilon}} \right)$. Then there exists   $\tilde{y}  \in \overline{\R^{n}_{-}} $ such that
$\frac{ \tilde{y}_{\epsilon} }{\nu_{\epsilon}} \rightarrow \tilde{y}$ as  $\epsilon \to 0$.  Therefore  $\tilde{w}(\tilde{y})= \lim_{\eps\to 0} \tilde{w}_{\epsilon}(\nu_{\epsilon}^{-1}\tilde{y}_{\epsilon})=1$. Therefore $\tilde{y}  \in  \R^{n}_{-}$, and then  $ \tilde{w} \in  C^{1}( \overline{\R^{n}_{-}}) $ is a nontrivial weak solution of the equation 
\begin{eqnarray*}
 \left \{ \begin{array} {lc}
          \Delta \tilde{w}  =  \tilde{w}^{2^{*}-1}   \qquad  \text{ in }  \R^{n}_{-}\\
        \quad  \tilde{w} =0  \qquad   \qquad \text{on }    \{ x_{1 } =0 \}
            \end{array} \right. 
\end{eqnarray*}
which  is impossible, see Struwe \cite{struwe} (Chapter III, Theorem 1.3) and the Liouville theorem  on half space.  Hence \eqref{first lim} holds. This completes the proof of Step 1.

\medskip \noindent {\bf Step 2:}  \label{blowup step 2}
Next, arguing similarly as in Step $1$ and using \eqref{first lim}, we get that 
\begin{align}{\label{second lim}}
\lim \limits_{\epsilon \to 0} \frac{d(y_{\epsilon}, \partial \Omega )}{\beta_{\epsilon}}=+\infty
\end{align}
We define $w_\epsilon$ as in \eqref{def:we:lemma}. We fix $\eta \in C^{\infty}_{c}(\R^{n})$. Then $\eta w_{\epsilon} \in \mathscr{D}^{1,2}(\R^{n})$ for $\epsilon >0$ small. Arguing as in Step 1, for any $\theta >0$, there exists  $C({\theta}) >0$ such that
\begin{align}{\label{b'dd on the integral 3}}
\int \limits_{\R^{n}} \left|  \nabla \left( \eta  w_{\epsilon} \right)\right|^{2}~ dx & \leq  \left( \frac{\nu_{\epsilon}}{\beta_{\epsilon}} \right)^{n-2} C(\theta) \left\| \nabla \eta \right\|^{2}_{L^{n}} \left( ~\int \limits_{\R^{n}} u_{\epsilon}^{2^{*}} ~ dx \right)^{\frac{n-2}{n}}  \notag \\
&+ (1+ \theta)   \left( \frac{\nu_{\epsilon}}{\beta_{\epsilon}} \right)^{n-2} \int \limits_{\R^{n}} \left( \eta\left(\frac{x-y_{\epsilon}}{\beta_{\epsilon}} \right) \right)^{2} \left| \nabla u_{\epsilon} \right|^{2} ~ dx.
\end{align}
Arguing as in Step 1, $( \eta {w}_{\epsilon})_\eps$ is uniformly bounded in $\mathscr{D}^{1,2} (\R^{n})$, and there exists $w  \in \mathscr{D}^{1,2} (\R^{n})$ such that upto a subsequence 
\begin{eqnarray}
 \left \{ \begin{array} {lc}
          \eta {w}_{\epsilon} \rightharpoonup  \eta {w} \qquad  \qquad \text{ weakly in }   \mathscr{D}^{1,2}(\R^{n}) ~  \text{ as } \epsilon \to 0 \\
          \eta {w}_{\epsilon}(x) \rightarrow  \eta{w}(x) \qquad  a.e ~ x ~ \text{ in } \R^{n} ~  \text{ as } \epsilon \to 0 
            \end{array} \right. 
\end{eqnarray}
Further $w \in C^{\infty} (\R^{n}) \cap \mathscr{D}^{1,2}(\R^{n})$, $w\geq 0$ and  it satisfies weakly the equation $\Delta w  =  {w}^{2^{*}-1}$ in $\rn$. Moreover $\lim_{\eps\to 0}w_{\epsilon}=w$ in $C^{1}_{loc}(\R^{n})$, $w(0)=1$ and $w>0$. This ends  Step $2$ and  proves Lemma \ref{blowup lemma}.
 \end{proof}

\medskip
\noindent
We let $(u_{\epsilon})$  be as in Theorem \ref{th:1}. We  will say  that{\emph{ blowup occurs}} whenever $u_{\epsilon}  \rightharpoonup  0$ weakly in $H^{2}_{1,0}(\Omega)$ as $\epsilon \rightarrow 0$. We describe the behaviour of such a sequence of solutions $(u_{\epsilon})$. By regularity, for all $\epsilon $,   $u_{\epsilon} \in C^{0}(\overline{\Omega})$. We let $x_{\epsilon} \in \Omega$  and $\mu_{\epsilon} >0$ be such that :
\begin{align}{\label{mu epsilon}}
u_{\epsilon}(x_{\epsilon})= \max \limits_{\overline{\Omega}} u_{\epsilon} (x) \qquad \text{ and } \qquad  \mu_{\epsilon}^{-\frac{n-2}{2}}= u_{\epsilon}(x_{\epsilon})
\end{align}
The main result of this section is the following theorem:

\begin{theorem}{\label{first blowup theorem}}
Let  $\Omega$ be  a  bounded smooth oriented  domain of $\R^{n}$,  $n \geq 3$ , such that $0 \in \partial \Omega$, and let  $a \in C^{1}(\overline{\Omega})$ be such that the operator  $\Delta + a$ is coercive in $\Omega$. Let $(s_{\epsilon})_{\epsilon >0} \in (0,2)$ be  a sequence such that $\lim \limits_{\epsilon \to 0 } s_{\epsilon}=0$. Suppose that the sequence  $\left(u_{\epsilon} \right)_{\epsilon>0} \in H^{2}_{1,0}(\Omega)$, where for  each  $\epsilon >0$,  $u_{\epsilon}$  satisfies  $(\ref{the eqn})$ and $(\ref{min energy condition})$, is   a  {\emph{blowup sequence}}, i.e 
\begin{align*}
u_{\epsilon}  \rightharpoonup  0  \qquad \text{weakly in }  H^{2}_{1,0}(\Omega) \qquad  \text{ as } ~\epsilon \rightarrow 0
\end{align*}
We let $(x_\epsilon)_\epsilon,(\mu_\epsilon)_\epsilon$ be as in \eqref{mu epsilon}. Let $k_{\epsilon}$ be such that
\begin{align}{\label{def k}}
k_{\epsilon}:= \left| x_{\epsilon} \right|^{s_{\epsilon}/2} \mu_{\epsilon}^{\frac{2-s_{\epsilon}}{2}} \qquad \text{ for } \epsilon >0
\end{align}
Then 
\begin{equation*}
\lim_{\epsilon\to 0}\mu_\epsilon=\lim_{\epsilon\to 0}k_\epsilon=0\hbox{ and }\lim_{\epsilon\to 0}\frac{d(x_\epsilon,\partial\Omega)}{\mu_\epsilon}=\lim_{\epsilon\to 0}\frac{d(x_\epsilon,\partial\Omega)}{k_\epsilon}=+\infty.
\end{equation*}
We rescale and define 
\begin{align*}
v_{\epsilon}(x):=  \frac{u_{\epsilon}(x_{\epsilon} + k_{\epsilon} x)}{u_{\epsilon}(x_{\epsilon})} \qquad \text{ for } x \in \frac{\Omega-x_{\epsilon}}{k_{\epsilon}}
\end{align*}
Then there exists $v \in C^{\infty} (\R^{n})$ such that $v\neq 0$ and for  any $\eta \in C^{\infty}_{c}(\R^{n})$
\begin{align*}
\eta v_{\epsilon} \rightharpoonup \eta v \hbox{ weakly in } \mathscr{D}^{1,2}(\R^n) \qquad \text{ as } \epsilon \rightarrow 0
\end{align*}
and $\lim_{\eps\to 0}v_\eps=v$ in $C^{1}_{loc}(\R^{n})$ where for $x\in\rn$,
\begin{align}{\label{bubble}}
v(x) =  \left( 1+ \frac{|x|^{2}}{n(n-2)}\right)^{-\frac{n-2}{2}}
 \hbox{ and } ~ \int_{\R^{n}}  |\nabla v|^{2} ~ dx= \left( \frac{1}{K(n,0)} \right)^{\frac{2^{*}}{2^{*}-2}}
\end{align}
Moreover upto a subsequence, as $\epsilon \to 0$ 
\begin{align}{\label{lim of scaling factor}}
\left( \frac{\mu_{\epsilon}}{|x_{\epsilon}|}  \right)^{s_{\epsilon}}    \to  1  \qquad \text{ and }~ \frac{k_{\epsilon}}{\mu_{\epsilon}}    \to  1.
\end{align}
\end{theorem}
\begin{proof} The proof goes through following steps.

\medskip\noindent {\bf Step 1:} We claim that: $\mu_{\epsilon}= o(1)$ as $ \epsilon \rightarrow 0$.

\medskip\noindent 
We prove our claim.  Suppose  $\lim \limits_{\epsilon \rightarrow 0} \mu_{\epsilon} \neq 0$. Then $(u_\epsilon)$ is uniformly bounded in $L^\infty$, and then  $(|x|^{-s} u_{\epsilon}^{2^{\star}(s_{\epsilon})-1} )_\epsilon$  is uniformly bounded in $L^{p}(\Omega) $ for some $p>n$. Then from $(\ref{the eqn})$, the weak convergence to $0$ and standard elliptic theory, we get that $u_{\epsilon} \rightarrow 0 $ in $C^{1}(\overline{\Omega})$, as $\epsilon \rightarrow 0$. From $(\ref{the eqn})$ and  $(\ref{min energy condition})$, we then get that $\lim \limits_{\epsilon \to 0} \mu_{s_{\epsilon},a}(\Omega)=0$ and therefore, $\mu_a(\Omega)=0$, contradicting the coercivity. This ends Step 1.

\medskip\noindent {\bf Step 2:}  From Lemma \ref{blowup lemma} it follows that 
\begin{align}{\label{blowup limits}}
\lim_{\epsilon\to 0}\frac{|x_\epsilon|}{\mu_\epsilon}=+\infty \text{ , } \lim_{\epsilon\to 0}\frac{d(x_\epsilon,\partial\Omega)}{k_\epsilon}=+\infty.
\end{align}
and, there exist $ {v} \in  C^{1}( \R^{n})$, $v > 0$, such that $\lim_{\eps\to 0}v_{\epsilon}=v$ in $ C^{1}_{loc}(\R^{n})$ and it satisfies $ \Delta v  =  {v}^{2^{*}-1} $ in $\rn$. Further   we have that $ \max_{ x\in \R^{n}} v(x)=v(0)=1$. By Caffarelli, Gidas and Spruck  \cite{cgs}, we then have the first assertion of \eqref{bubble}.

\medskip\noindent{\bf Step 3: } Arguing as in the proof of $(\ref{b'dd on the integral 3})$, for any $\theta>0$, there exists $C(\theta)>0$ such that for any $R>0$
\begin{align}{\label{b'dd on the integral 3 b}}
\int \limits_{\R^{n}} \left|  \nabla (\eta_{R}   v_{\epsilon}) \right|^{2}~ dx \leq   C(\theta)   \left( ~\int_{B_{0}(2R) \backslash B_{0}(R)} (\eta_{2R}   v_{\epsilon})^{2^{*}} ~ dx \right)^{\frac{n-2}{n}}  + ( 1+ \theta)  \left( \frac{\mu_{\epsilon}}{k_{\epsilon}} \right)^{n-2}   \int \limits_{\Omega}  \left| \nabla u_{\epsilon} \right|^{2} ~ dx
\end{align}
Now $u_{\epsilon}  \rightharpoonup  0 $ weakly in   $H^{2}_{1,0}(\Omega)$  as $\epsilon \rightarrow 0$, where for  each  $\epsilon >0$,  $u_{\epsilon}$  satisfies  $(\ref{the eqn})$ and $(\ref{min energy condition})$. So  we have
\begin{align*}
\int \limits_{\Omega}  \left| \nabla u_{\epsilon} \right|^{2} ~ dx =  \int \limits_{\Omega }  \frac{|u_{\epsilon}(x)|^{2^{*}(s_{\epsilon})}}{|x|^{s_{\epsilon}}} ~ dx + o(1)  \leq \mu_{s_\eps,a}(\Omega)^{\frac{2^{*}(s_{\epsilon})}{2^{*}(\epsilon)-2}} + o(1) \hbox{ as } \epsilon \to 0. 
\end{align*}
Using Proposition \ref{lim of mu}, letting $\epsilon \to 0$, then $R \to + \infty$, and the $\theta\to 0$,  we  obtain 
\begin{align}{\label{b'dd on integral c}}
\int \limits_{\R^{n}} \left|  \nabla  v \right|^{2}~ dx  \leq \left(  \limsup \limits_{\epsilon \to 0} \left( \frac{\mu_{\epsilon}}{|x_{\epsilon}|}  \right)^{s_{\epsilon}}   \right)^{{\frac{n-2}{2}} }  \mu_{a}(\Omega)^{\frac{2^{*}}{2^{*}-2}} 
\end{align} 
From  \eqref{blowup limits}  we get $ \limsup \limits_{\epsilon \to 0} \left( \frac{\mu_{\epsilon}}{|x_{\epsilon}|}  \right)^{s_{\epsilon}}  \leq 1$. Since $\mu_{a}(\Omega ) \leq \frac{1}{K(n,0)}$ (see Aubin \cite{aubin}), we get
\begin{align*}
\int \limits_{\R^{n}} \left|  \nabla  v \right|^{2}~ dx \leq  \mu_{a}(\Omega)^{\frac{2^{*}}{2^{*}-2}} \leq \left(\frac{1}{K(n,0)} \right)^{\frac{2^{*}}{2^{*}-2}} 
\end{align*}
Since $v\in \mathscr{D}^{1,2}(\R^n)$ satisfies $\Delta v  =  {v}^{2^{*}-1}$, Sobolev's inequality \eqref{best Sobolev  constant} then yields the seocnd assertion of \eqref{bubble}. Then   $(\ref{b'dd on integral c})$  implies $\limsup\limits_{\epsilon \to 0} \left( \frac{\mu_{\epsilon}}{|x_{\epsilon}|}  \right)^{s_{\epsilon}}    \geq 1$, which yields \eqref{lim of scaling factor}. This completes the proof of Theorem \ref{first blowup theorem}.
\end{proof}

\medskip\noindent As a consequence of Theorem \ref{first blowup theorem}, we get the following concentration of energy:

\begin{proposition}{\label{conc of energy}}
Under the hypothesis of  Theorem \ref{first blowup theorem} one further has that
\begin{align*}
\lim \limits_{R \to +\infty} \lim \limits_{\epsilon \to 0} \int \limits_{\Omega \backslash B_{x_{\epsilon}}(R k_{\epsilon}) }  \frac{|u_{\epsilon}(x)|^{2^{*}(s_{\epsilon})}}{|x|^{s_{\epsilon}}} ~ dx =0
\end{align*}
\end{proposition}
\begin{proof}
We obtain by change of variables 
\begin{align*}
 \int \limits_{\Omega \backslash B_{x_{\epsilon}}(R k_{\epsilon}) }  \frac{|u_{\epsilon}(x)|^{2^{*}(s_{\epsilon})}}{|x|^{s_{\epsilon}}} ~ dx =& \int \limits_{\Omega }  \frac{|u_{\epsilon}(x)|^{2^{*}(s_{\epsilon})}}{|x|^{s_{\epsilon}}} ~ dx -  \int \limits_{ B_{x_{\epsilon}}(R k_{\epsilon}) }  \frac{|u_{\epsilon}(x)|^{2^{*}(s_{\epsilon})}}{|x|^{s_{\epsilon}}} ~ dx \\
 =& \int \limits_{\Omega }  \frac{|u_{\epsilon}(x)|^{2^{*}(s_{\epsilon})}}{|x|^{s_{\epsilon}}} ~ dx - \frac{k_{\epsilon}^{n}}{\mu_{\epsilon}^{n-s_{\epsilon}}}  \int \limits_{ B_{0}(R) }  \frac{|v_{\epsilon}(x)|^{2^{*}(s_{\epsilon})}}{ \left| x_{\epsilon} + k_{\epsilon}x \right|^{s_{\epsilon}}} ~ dx \\
  =& \int \limits_{\Omega }  \frac{|u_{\epsilon}(x)|^{2^{*}(s_{\epsilon})}}{|x|^{s_{\epsilon}}} ~ dx - \left(  \frac{ \left|x_{\epsilon}\right|^{s_{\epsilon}}}{\mu_{\epsilon}^{s_{\epsilon}}} \right) ^{\frac{n-2}{2}} \int \limits_{ B_{0}(R) }  \frac{|v_{\epsilon}(x)|^{2^{*}(s_{\epsilon})}}{ \left| \frac{x_{\epsilon}}{|x_{\epsilon}|} + \frac{k_{\epsilon}}{|x_{\epsilon}|}x \right|^{s_{\epsilon}}} ~ dx 
\end{align*}
Letting $\epsilon \to 0$ and then $R \to + \infty$ one obtains the proposition using Theorem \ref{first blowup theorem}.\end{proof}

\section{Refined Blowup Analysis I}\label{sec:blowup:1}
\noindent
In this section we obtain pointwise bounds  on the  blowup sequence $(u_{\epsilon})_{\epsilon>0}$ that will be used in next section to get the optimal bound.

\begin{theorem}\label{th:bu1} With the same hypothesis as in Theorem \ref{first blowup theorem}, we have that  there exists a constant $C>0$ such that  for $\epsilon >0$
\begin{align*}
\left|x-x_{\epsilon} \right|^{\frac{n-2}{2}} u_{\epsilon}(x)+\frac{\left|x-x_{\epsilon} \right|^{\frac{n}{2}}}{d(x,\partial\Omega)} u_{\epsilon}(x)\leq C \qquad \text{ for  all } x \in \Omega.
\end{align*}
Moreover,
\begin{align*}
\lim \limits_{R \to + \infty} \lim \limits_{\epsilon \to 0} \sup \limits_{x \in \Omega \backslash B_{x_{\epsilon}}(Rk_{\epsilon}) }  
\left|x-x_{\epsilon} \right|^{\frac{n-2}{2}} u_{\epsilon}(x) =0
\end{align*}
\end{theorem}
\noindent
The proof of Theorem \ref{th:bu1} comprises the three propositions proved below.

\begin{proposition}{\label{strong b'dds lemma one}}
With the same hypothesis as in Theorem \ref{first blowup theorem}, we have that  there exists a constant $C>0$ such that  for $\epsilon >0$
\begin{align*}
\left|x-x_{\epsilon} \right|^{\frac{n-2}{2}} u_{\epsilon}(x)\leq C \qquad \text{ for  all } x \in \Omega
\end{align*}
\end{proposition}
\begin{proof} We argue by contradiction and let $y_{\epsilon} \in \Omega$ be such that 
\begin{align}\label{def:ye}
 \left|y_{\epsilon}-x_{\epsilon} \right|^{\frac{n-2}{2}} u_{\epsilon}(y_{\epsilon}) =\sup \limits_{x \in  \Omega} \left(  \left|x-x_{\epsilon} \right|^{\frac{n-2}{2}} u_{\epsilon}(x)  \right)\to +\infty\hbox{ as }\eps\to 0.
\end{align}
Then 
\begin{align}{\label{on the contrary 1, strong b'dds lemma one}}
  \left|y_{\epsilon}-x_{\epsilon} \right|^{\frac{n-2}{2}} u_{\epsilon}(y_{\epsilon})  \longrightarrow + \infty \qquad \text{ as } \epsilon \to 0
\end{align}
We define $\lambda_{\epsilon}^{-\frac{n-2}{2}:} =  u_{\epsilon}(y_{\epsilon})$. Then  $\mu_{\epsilon} \leq \lambda_{\epsilon}$, and  $(\ref{on the contrary 1, strong b'dds lemma one})$ becomes 
\begin{equation}\label{lim:y:x:l}
\lim \limits_{\epsilon \to 0} \frac{\left| y_{\epsilon}-x_{\epsilon}\right|}{\lambda_{\epsilon}} =+ \infty.
\end{equation}
and so we  have that $\displaystyle{\lim \limits_{\epsilon \to 0} \lambda_{\epsilon} =0}$.

\medskip
\noindent {\bf Step 1:} It follows from the definition \eqref{def:ye} and \eqref{lim:y:x:l} that given  any $R>0$  one has for $\epsilon >0$ sufficiently small $u_{\epsilon}(x) \leq  2 u_{\epsilon}(y_{\epsilon})$ for  all $x \in B_{y_{\epsilon}}(R \lambda_{\epsilon})$. Therefore hypothesis  \eqref{bdds on u 1} of Lemma \ref{blowup lemma} is satisfied and one  has 
 \begin{align}\label{step 1, strong b'dds lemma one}
\lim \limits_{\epsilon \rightarrow 0} \frac{|y_{\epsilon}|}{\lambda_{\epsilon}}= + \infty
\end{align}
\noindent  Define $l_{\epsilon}= |y_{\epsilon}|^{s_{\epsilon}/2} \lambda_{\epsilon}^{\frac{2-s_{\epsilon}}{2}}$ for all $\eps>0$.  Then   $\displaystyle{\lim \limits_{\epsilon \to 0} l_{\epsilon} =0}$. Moreover, we have that
\begin{equation}\label{lim:l:y}
\lim \limits_{\epsilon \to 0} \frac{l_{\epsilon}}{|y_{\epsilon}|}= \lim \limits_{\epsilon \to 0} \left( \frac{\lambda_{\epsilon}}{|y_{\epsilon}|} \right)^{\frac{2-s_{\epsilon}}{2}} =0.
\end{equation}

\medskip\noindent{\bf Step 2:} We claim that 
\begin{equation}\label{lim:x:y:l}
\lim_{\epsilon\to 0}\frac{|y_\epsilon-x_\epsilon|}{l_\epsilon}=+\infty.
\end{equation}
We prove the claim. Due to \eqref{lim:l:y}, the claim is clear when $y_\epsilon=O(|y_\epsilon-x_\epsilon|)$ as $\epsilon\to 0$. We assume that $y_\epsilon-x_\epsilon=o(|y_\epsilon|)$ as $\eps\to 0$. We then have that $|x_\eps|\asymp|y_\eps|$ as $\eps\to 0$. Therefore, there exists $c_0>0$ such that
\begin{align}\label{id:1}
c_0\leq \frac{|x_{\epsilon}|^{s_{\epsilon}}}{|y_{\epsilon}|^{s_{\epsilon}}}= \frac{\lambda_{\epsilon}^{s_{\epsilon}}}{|y_{\epsilon}|^{s_{\epsilon}}}   \frac{|x_{\epsilon}|^{s_{\epsilon}}}{\lambda_{\epsilon}^{s_{\epsilon}}} \leq  \frac{\lambda_{\epsilon}^{s_{\epsilon}}}{|y_{\epsilon}|^{s_{\epsilon}}}   \frac{|x_{\epsilon}|^{s_{\epsilon}}}{\mu_{\epsilon}^{s_{\epsilon}}}
\end{align}
Since $\lim \limits_{\epsilon \to 0 }  \frac{|x_{\epsilon}|^{s_{\epsilon}}}{\mu_{\epsilon}^{s_{\epsilon}}}=1$ as shown in \eqref{lim of scaling factor}, it follows that there exists $c_1>0$ such that $\left(\frac{\lambda_\eps}{|y_\eps|}\right)^{s_\eps}\geq c_1$ for $\eps>0$ small enough. Therefore,
\begin{align}
 \lim_{\epsilon \to 0} \frac{\left| y_{\epsilon}-x_{\epsilon}\right|}{l_{\epsilon}} =  \lim_{\epsilon \to 0} \frac{\left| y_{\epsilon}-x_{\epsilon}\right|}{\lambda_{\epsilon}}  \frac{\lambda_{\epsilon}^{s_{\epsilon}/2 } }{ |y_{\epsilon}|^{s_{\epsilon}/2}}= + \infty
\end{align}
and the claim is proved.

\medskip\noindent{\bf Step 3:} It follows from \eqref{lim:x:y:l} and the definitions \eqref{def:ye} and \eqref{on the contrary 1, strong b'dds lemma one} that for any $R>0$  one has for $\epsilon >0$ sufficiently small $ u_{\epsilon}(x) \leq  2 u_{\epsilon}(y_{\epsilon})$ for all $x \in B_{y_{\epsilon}}(R l_{\epsilon})$. Therefore hypothesis  \eqref{bdds on u 2} of Lemma \ref{blowup lemma} is satisfied and one  has 
\begin{align}\label{lim:d:l:2}
\lim \limits_{\epsilon \to 0} \frac{d \left( y_{\epsilon}, \partial \Omega\right)}{l_{\epsilon}}= + \infty.
\end{align}

\noindent
We  let  for  $\epsilon >0$ 
\begin{align*}
w_{\epsilon}(x)=  \frac{u_{\epsilon}(y_{\epsilon} + l_{\epsilon} x)}{u_{\epsilon}(y_{\epsilon})} \qquad \text{ for } x \in \frac{\Omega-y_{\epsilon}}{l_{\epsilon}}.
\end{align*}
From Lemma \ref{blowup lemma} it follows that $\lim_{\eps\to 0}w_{\epsilon}= w$ in  $C^{1}_{loc}(\R^{n})$ where ${w} \in   C^{\infty} (\R^{n}) \cap \mathscr{D}^{1,2}(\R^{n}) $ is  such that $\Delta w  =  {w}^{2^{*}-1} $ in $\R^{n}$
$w \geq 0$ and $w(0)=1$. We obtain by a change of variable  for $R>0$ and $\epsilon >0$
\begin{align*}
\int \limits_{ B_{0}( R) }  \frac{|w_{\epsilon}(x)|^{2^{*}(s_{\epsilon})}}{ \left| \frac{y_{\epsilon}}{|y_{\epsilon}|} + \frac{l_{\epsilon}}{|y_{\epsilon}|}x \right|^{s_{\epsilon}}} ~ dx  = \left(  \frac{\lambda_{\epsilon}^{s_{\epsilon}} }{ \left|y_{\epsilon}\right|^{s_{\epsilon}} } \right)^{\frac{n-2}{2}}   \int \limits_{ B_{y_{\epsilon}}(  R l_{\epsilon} ) }  \frac{|u_{\epsilon}(x)|^{2^{*}(s_{\epsilon})}}{|x|^{s_{\epsilon}}} ~ dx 
\end{align*} 
Passing to  the limit  as $\epsilon \to 0$, we have  for $R>0$
\begin{align*}
\int \limits_{ B_{0}( R ) } w^{2^{*}} ~ dx  \leq
\limsup \limits_{\epsilon \to 0} \int \limits_{ B_{y_{\epsilon}}( R l_{\epsilon}) }  \frac{|u_{\epsilon}(x)|^{2^{*}(s_{\epsilon})}}{|x|^{s_{\epsilon}}} ~ dx 
\end{align*}
and so
\begin{align*}
  \int \limits_{ \R^{n}} w^{2^{*}} ~ dx   = \lim \limits_{R \to +\infty} \int \limits_{ B_{0}( R ) } w^{2^{*}} ~ dx  \leq  \lim \limits_{R \to +\infty}  \limsup \limits_{\epsilon \to 0} \int \limits_{ B_{y_{\epsilon}}( R l_{\epsilon}) }  \frac{|u_{\epsilon}(x)|^{2^{*}(s_{\epsilon})}}{|x|^{s_{\epsilon}}} ~ dx 
\end{align*}
Now for any $R >0$, we claim that $\displaystyle{  B_{x_{\epsilon}}(R k_{\epsilon}) \cap  B_{y_{\epsilon}}(R l_{\epsilon})=\emptyset}$  for  $\epsilon >0$ sufficiently small.  We argue by contardiction and we assume that the intersection is nonempty, which yields $y_\eps-x_\eps=O(k_\eps+l_\eps)$ as $\eps\to 0$, up to extraction. It then follows from \eqref{lim:x:y:l} that $y_\eps-x_\eps=O(k_\eps)$ as $\eps\to 0$, and then $|y_\eps-x_\eps|^{\frac{n-2}{2}}u_\eps(y_\eps)=O(k_\eps^{\frac{n-2}{2}}u_\eps(y_\eps))=O(\mu_\eps^{\frac{n-2}{2}}u_\eps(x_\eps))=O(1)$ with \eqref{mu epsilon} and \eqref{lim of scaling factor}. This contradicts \eqref{on the contrary 1, strong b'dds lemma one} and proves the claim. Then by Proposition  \ref{conc of energy}
\begin{align*}
  \int \limits_{ \R^{n}} w^{2^{*}} ~ dx  \leq 
\lim \limits_{R \to +\infty} \limsup \limits_{\epsilon \to 0} \int \limits_{\Omega \backslash B_{y_{\epsilon}}(R l_{\epsilon}) }  \frac{|u_{\epsilon}(x)|^{2^{*}(s_{\epsilon})}}{|x|^{s_{\epsilon}}} ~ dx =0
\end{align*}
A contradiction since $w(0)=1$. Hence \eqref{on the contrary 1, strong b'dds lemma one} does not hold. This completes the proof of Proposition \ref{strong b'dds lemma one}.\end{proof}

\noindent Having obtained the strong bound in Proposition \ref{strong b'dds lemma one} we show that

\begin{proposition}{\label{strong b'dds lemma two}}
With the same hypothesis as in Theorem \ref{first blowup theorem} we have that there exists a constant $C>0$ such that  for $\epsilon >0$
\begin{align*}
\left|x-x_{\epsilon} \right|^{{n}/{2}}  |\nabla u_{\epsilon}(x)| \leq C  ~ \text{ and }~  \left|x-x_{\epsilon} \right|^{n/2} u_{\epsilon}(x) \leq C d(x, \partial \Omega) \qquad \text{ for  all } x \in \Omega
\end{align*}
\end{proposition}
\begin{proof}
We proceed by contradiction and assume that there exists a sequence of points $(y_{\epsilon})_{\epsilon >0}$ in $\Omega $ such that 
\begin{align}{\label{on the contrary, strong b'dds lemma two}}
\left|y_{\epsilon} -x_{\epsilon} \right|^{{n}/{2}}  |\nabla u_{\epsilon}(y_{\epsilon} )| + \frac{\left|y_{\epsilon} -x_{\epsilon} \right|^{n/2} u_{\epsilon}(y_{\epsilon})}{d(y_{\epsilon} , \partial \Omega)} \longrightarrow + \infty  \qquad \text{ as } \epsilon \to 0 
\end{align}
We define $ \lim_{\epsilon \to 0} x_{\epsilon} =x_{0} \in \overline{\Omega}$ and $\lim_{\epsilon \to 0} y_{\epsilon} =y_{0} \in \overline{\Omega}$.

\medskip\noindent{\bf Case 1:} we assume that $x_{0} \neq y_{0}$. We  choose   $\delta >0$ such that  $ 0 < 4 \delta < {\left| x_{0}- y_{0} \right|}$. Then one has that $\delta < \left| x- x_{\epsilon} \right|$ for all $ x \in B_{y_{0}}(2\delta) \cap \Omega$ and  Lemma \ref{strong b'dds lemma one}  then gives us that there exists a constant $C(\delta) >0$ such that $0 \leq u_{\epsilon} \leq C(\delta)$ in $ B_{y_{0}}(2\delta)$. Then from equation \eqref{the eqn} and standard elliptic theory, $u_\eps$  is bounded in $ C^{1}\left(  B_{y_{0}}(\delta)  \cap \overline{\Omega} \right)$.  So there exists a constant $C >0$ such that $|\nabla u_{\epsilon}(x)| \leq C$ and $u_{\epsilon} (x) \leq C d(x, \partial \Omega)$ for all $x \in B_{y_{0}}(\delta) \cap \overline{\Omega}$. This contradictis \eqref{on the contrary, strong b'dds lemma two}. The proposition is proved in Case 1. 

\medskip\noindent{\bf Case 2:} we assume that $x_{0} = y_{0}$. Define $\alpha_{\epsilon} = \left| y_{\epsilon}- x_{\epsilon} \right|$, so that $\lim \limits_{\epsilon \to 0} \alpha_{\epsilon}=0$. 

\medskip\noindent{\bf Case 2.1:} We assume that upto a subsequence 
\begin{align*}
d(x_{\epsilon}, \partial \Omega)  \geq 2  \left| y_{\epsilon}- x_{\epsilon} \right|
\end{align*}
For $\epsilon >0$ we let 
\begin{align*}
\tilde{u}_{\epsilon}(x)= \alpha_{\epsilon}^{\frac{n-2}{2}} u_{\epsilon} \left( x_{\epsilon}  + \alpha_{\epsilon} x \right)   \qquad \text{ for } ~ x \in  B_{0}(3/2)
\end{align*}
This is well defined since $B_{x_{\epsilon}}(2 \alpha_{\epsilon}  ) \subset \Omega$. Using Lemma  \ref{strong b'dds lemma one}  one  obtains that there exists a  constant $C>0$ such that 
\begin{align*}
|x|^{\frac{n-2}{2}} \tilde{u}_{\epsilon}(x) \leq C \qquad \text{ for } x \in  B_{0}(3/2).
\end{align*}
Arguing as in Step 1.3 of the proof of Lemma \ref{strong b'dds lemma one}, standard elliptic theory yields
\begin{align*}
\left\| \tilde{u}_{\epsilon} \right\|_{C^{1}\left(  B_{0}(5/4) \setminus \overline{B_{0}(1/2) } \right)} =O(1) \qquad \text{ as } \epsilon \to 0
\end{align*}
Then one then obtains  as  $\epsilon \to 0$
\begin{align*}
\left| \nabla \tilde{u}_{\epsilon} \left( \frac{ y_{\epsilon} - x_{\epsilon}}{ | y_{\epsilon}-x_{\epsilon} |} \right) \right| =O(1) \hbox{ and }   \tilde{u}_{\epsilon} \left( \frac{ y_{\epsilon} - x_{\epsilon}}{ | y_{\epsilon}-x_{\epsilon} |} \right)  = O(1). 
\end{align*}
coming back to the definition of $\tilde{u}_{\epsilon}$, this contradicts \eqref{on the contrary, strong b'dds lemma two}. This ends Case 2.1.

\medskip\noindent{\bf Case 2.2:} We assume that upto a subsequence 
\begin{align*}
d(x_{\epsilon}, \partial \Omega)  \leq 2  \left| y_{\epsilon}- x_{\epsilon} \right|
\end{align*}
Let $\displaystyle{\mathcal{T}: U \to V  }$ be a parametrisation of the boundary $\partial \Omega$ as in \eqref{def:T:bdy} around the point $p=x_{0}$.
Let $z_{\epsilon} \in \partial \Omega$ be such that $\left|  z_{\epsilon} - x_{\epsilon} \right| = d(x_{\epsilon}, \partial \Omega)$ for $\epsilon >0$. We let $\tilde{x}_{\epsilon}$, $\tilde{z}_{\epsilon} \in U$  be such that  $\mathcal{T}(\tilde{x}_{\epsilon})= {x}_{\epsilon}$ and $\mathcal{T}(\tilde{z}_{\epsilon})= {z}_{\epsilon} $. Then it follows from the properties of the boundary chart $\mathcal{T}$, that $\lim_{\epsilon \to 0} \tilde{x}_{\epsilon}=0= \lim_{\epsilon \to 0} \tilde{z}_{\epsilon}$, $(\tilde{x}_{\epsilon})_{1} <0 $ and $(\tilde{z}_{\epsilon})_{1} =0$. For all $\epsilon >0$, we let 
\begin{align*}
\tilde{u}_{\epsilon}(x) = \alpha^{\frac{n-2}{2}}_{\epsilon} u_{\epsilon} \circ \mathcal{T}(\tilde{z}_{\epsilon}+\alpha_{\epsilon  } x) \qquad \text{ for } x \in \frac{U -\tilde{z}_{\epsilon} }{\alpha_{\epsilon}}\cap \{ x_{1} \leq 0 \} 
\end{align*}
For any $R >0$,  $\tilde{u}_{\epsilon}$ is defined in $B_{0}(R) \cap \{ x_{1} \leq 0 \} $ for $\epsilon >0$ small enough. Using lemma  Lemma \ref{strong b'dds lemma one} and the properties of the chart $\mathcal{T}$, one  obtains that there exists a  constant $C>0$ such that 
\begin{align*}
  \left| \rho_{\epsilon}  -x\right|^{\frac{n-2}{2}} \tilde{u}_{\epsilon}(x) \leq C \qquad \text{ for } x \in B_{0}(R) \cap \{ x_{1} \leq 0 \} 
\end{align*}
where $\rho_{\epsilon}=  \frac{\tilde{x}_{\epsilon}- \tilde{z}_{\epsilon} }{\alpha_{\epsilon}}$ and there exists $\rho_{0} \in \overline{\R}_{-}$ such that $\rho_{\epsilon}  \rightarrow \rho_{0}$ as $\eps\to 0$. Arguing again as in Step 1.3 of the proof of Lemma \ref{blowup lemma},  standard elliptic theory yields 
 \begin{align*}
 \left\| \tilde{u}_{\epsilon} \right\|_{C^{1}\left( \overline{B_{0}(R/2)}\setminus{{B_{\rho_{0}}(2\delta)}} \cap \{ x_{1} \leq 0 \} \right)} =O(1) \qquad \text{ as } \epsilon \to 0
 \end{align*}
 and $\tilde{u}_{\epsilon}$ vanishes on the boundary $B_{0}(R/2)\setminus{\overline{B_{\rho_{0}}(2\delta)}} \cap    \{ x_{1 } =0 \}$. 
Let $\tilde{y}_{\epsilon} \in U$ be such that  $\mathcal{T}(\tilde{y}_{\epsilon})= y_{\epsilon}$. It then follows that, as $\epsilon \to 0$
\begin{align*}
\left| \nabla \tilde{u}_{\epsilon} \left( \ \frac{\tilde{y}_{\epsilon}-\tilde{z}_{\epsilon}}{\alpha_{\epsilon} } \right) \right| =O(1), \qquad   \tilde{u}_{\epsilon} \left(  \frac{\tilde{y}_{\epsilon}-\tilde{z}_{\epsilon}}{\alpha_{\epsilon} } \right)  = O(1) \qquad 
\end{align*}
and since $\tilde{u}_{\epsilon}$ vanishes on the boundary $B_{0}(R/2)\setminus{\overline{B_{\rho_{0}}(2\delta)}} \cap    \{ x_{1 } =0 \}$, it follows that 
\begin{align*}
0 \leq  \tilde{u}_{\epsilon} \left(  \frac{\tilde{y}_{\epsilon}-\tilde{z}_{\epsilon}}{\alpha_{\epsilon} } \right) =  O \left(\frac{(\tilde{y}_{\epsilon}- \tilde{z}_{\epsilon})_{1}}{\alpha_{\epsilon}  } \right)=O \left(\frac{(\tilde{y}_{\epsilon})_{1}}{\alpha_{\epsilon}  } \right) =  O \left(\frac{d(y_{\epsilon}, \partial \Omega )}{\alpha_{\epsilon}  } \right) 
\end{align*}
comig back to the definition of $\tilde{u}_{\epsilon}$ this  implies that as  $\epsilon \to 0$
$$\left|y_{\epsilon} -x_{\epsilon} \right|^{{n}/{2}}  |\nabla u_{\epsilon}(y_{\epsilon} )| = O(1), \hbox{ and }
 \left|y_{\epsilon} -x_{\epsilon} \right|^{{n}/{2}}  u_{\epsilon}(y_{\epsilon} )   =O(d(x_{\epsilon}, \partial \Omega)),$$
contradicting  \eqref{on the contrary, strong b'dds lemma two}. This ends Case 2.2.

\medskip\noindent All these cases prove Proposition \ref{strong b'dds lemma two}.\end{proof}

\medskip\noindent As a consequence of Proposition \ref{strong b'dds lemma one} and   Proposition \ref{strong b'dds lemma two} we get the following:
\begin{corollary}{\label{u strongly goes to 0 locally}}
Let $\left(u_{\epsilon} \right)_{\epsilon >0}$ be  as in Theorem \ref{first blowup theorem}, and let $\lim \limits_{\epsilon \to 0} x_{\epsilon} \to x_{0} \in \overline{\Omega}$, then  upto a subsequence $\lim_{\epsilon \to 0} u_{\epsilon}=0$ in $C^{1}_{loc}(\overline{\Omega}\backslash \{ x_{0}\})$.
\end{corollary}
We slightly improve our estimate  in Proposition \ref{strong b'dds lemma one}   to  obtain  
\begin{proposition}{\label{strong conc estimate prop}}
With the same hypothesis as in Theorem \ref{first blowup theorem}  we have
\begin{align*}
\lim \limits_{R \to + \infty} \lim \limits_{\epsilon \to 0} \sup \limits_{x \in \Omega \backslash B_{x_{\epsilon}}(Rk_{\epsilon}) }  
\left|x-x_{\epsilon} \right|^{\frac{n-2}{2}} u_{\epsilon}(x) =0
\end{align*}
\end{proposition}
\begin{proof}
Suppose on the contrary there exists $\epsilon_{0}>0$   and a sequence of  points $(y_{\epsilon})_{\epsilon>0} \in \Omega $ such that upto a subsequence 
\begin{align}{\label{strong conc estimate contrary 1}}
 \left|y_{\epsilon}-x_{\epsilon} \right|^{\frac{n-2}{2}} u_{\epsilon}( y_{\epsilon}  ) \geq \epsilon^{\frac{n-2}{2}}_{0} \qquad  \text{ and  } \qquad \lim  \limits_{\epsilon \to 0} \frac{\left|y_{\epsilon}-x_{\epsilon} \right|}{ k_{\epsilon}   } = + \infty
\end{align}
It then follows  from Corollary \ref{u strongly goes to 0 locally} that $\lim  \limits_{\epsilon \to 0} \left|y_{\epsilon}-x_{\epsilon} \right|=0$. We define $\lambda_{\epsilon}^{-\frac{n-2}{2}} =  u_{\epsilon}(y_{\epsilon})$. Then $(\ref{strong conc estimate contrary 1})$ becomes 
\begin{align}{\label{strong conc estimate contrary 2}}
C \geq \frac{\left| y_{\epsilon}-x_{\epsilon}\right|}{\lambda_{\epsilon}}   \geq \epsilon_{0} \qquad  \text{ for all }~ \epsilon >0
\end{align}
and so  $\displaystyle{\lim \limits_{\epsilon \to 0} \lambda_{\epsilon} =0}$.  using Lemma \ref{strong b'dds lemma one} we obtain that  as $\epsilon \to 0$
\begin{align}{\label{strong conc estimate contrary 4 a}}
 \frac{k_{\epsilon}}{\lambda_{\epsilon}}=  \frac{k_{\epsilon}}{|y_{\epsilon}-x_{\epsilon}|} \frac{|y_{\epsilon} - x_{\epsilon} |}{\lambda_{\epsilon}}= O\left( \frac{k_{\epsilon}}{|y_{\epsilon}-x_{\epsilon}|} \right) =o(1)
\end{align}
We define $l_{\epsilon}= |y_{\epsilon}|^{s_{\epsilon}/2} \lambda_{\epsilon}^{\frac{2-s_{\epsilon}}{2}}$ for $\epsilon >0$. Then $\displaystyle{\lim \limits_{\epsilon \to 0} l_{\epsilon} =0}$.

\medskip\noindent We first claim that 
\begin{align}\label{strong conc estimate prop, step 1}
\frac{|y_{\epsilon}|^{s_{\epsilon}}}{\lambda_{\epsilon}^{s_{\epsilon}}}=O(1) \qquad \text{ as } \epsilon \to 0 
\end{align}
We proceed by contradiction and we assume that $\lim \limits_{\epsilon \to 0 } \frac{\lambda_{\epsilon}^{s_{\epsilon}}}{|y_{\epsilon}|^{s_{\epsilon}}} =0$. Now, using \eqref{id:1} and \eqref{lim of scaling factor}, we get that $\lim \limits_{\epsilon \to 0 } \frac{|x_{\epsilon}|^{s_{\epsilon}}}{|y_{\epsilon}|^{s_{\epsilon}}} =0$. And in particular one has that $\lim \limits_{\epsilon \to 0 } \frac{|x_{\epsilon}|}{|y_{\epsilon}|}=0$ and  $\lim \limits_{\epsilon \to 0 } \frac{\lambda_{\epsilon}}{|y_{\epsilon}|}=0$.  Then
\begin{align*}
\frac{\left| y_{\epsilon}-x_{\epsilon}\right|}{\lambda_{\epsilon}}   \geq \frac{|y_{\epsilon}|}{\lambda_{\epsilon}} \left| 1- \frac{|x_{\epsilon}|}{|y_{\epsilon}|}\right| \to + \infty\hbox{ as }\epsilon\to 0
 \end{align*}
A contradiction to \eqref{strong conc estimate contrary 2}, proving our claim.  We note  that  then there exists $c_{2} > 0$ such that for $\epsilon>0$ small
\begin{align}{\label{strong conc estimate contrary 3}}
\frac{\left| y_{\epsilon}-x_{\epsilon}\right|}{l_{\epsilon}} = \frac{\left| y_{\epsilon}-x_{\epsilon}\right|}{\lambda_{\epsilon}} \frac{\lambda_{\epsilon}^{s_{\epsilon}/2}}{|y_{\epsilon}|^{s_{\epsilon}/2}}  \geq c_{2} 
\end{align}
Arguing as in case $2.2$ of Lemma \ref{strong b'dds lemma one} we  see that we  cannot have $\displaystyle{\lim \limits_{\epsilon \to 0} \frac{d(y_{\epsilon}, \partial \Omega )}{l_{\epsilon}}=0}$.  Let $\rho_{0}>0$  be such that upto a subsequence  $\displaystyle{\frac{d(y_{\epsilon}, \partial \Omega )}{l_{\epsilon}} \geq 2 \rho_{0} }$. Without loss of generality we can take $2\rho_{0} < c_{2}$.  Then proceeding as in  step $3$ of Lemma \ref{strong b'dds lemma one} we arrive at a contradiction. These  steps complete the proof of Proposition \ref{strong conc estimate prop}. \end{proof}

\section{Refined Blowup Analysis II}\label{sec:blowup:2}
This section is devoted to the proof of Theorem \ref{th:1}.
\begin{proof}{\bf Step 1:} We claim that for any $\alpha \in  (0, n-2)$, there exists  $C_{\alpha}>0$  such that  for all $\epsilon >0$
\begin{align}\label{strong estimate first approx}
\left| x-x_{\epsilon} \right|^{\alpha} \mu_{\epsilon}^{\frac{n-2}{2} - \alpha} u_{\epsilon}(x) \leq C_{\alpha} \qquad \text{ for all } x \in  \Omega
\end{align}
\begin{proof}
Since the operator $\Delta + a  $ is coercive  on $\Omega$ and $a \in C(\overline{\Omega})$, there exists $U_{0} \subset \R^{n}$ an open set such that $\overline{\Omega} \subset \subset U_{0}$, and there exists $a_{1}>0$, $A_{1}>0$ such that 
\begin{align*}
\int \limits_{U_{0}}  \left| \nabla \varphi \right|^{2} ~ dx+ \int \limits_{U_{0}}   \left(  a -  a_{1} \right)  \varphi^{2} ~ dx \geq  A_{1}  \int \limits_{U_{0}} \varphi^{2}~ dx  \qquad \text{ for all } \varphi \in C^{\infty}_{c}(U_{0}),
\end{align*}  
where we have continuously extended $a$ to $U_0$. In other words  the operator  $\Delta + (  a- a_{1})  $ is   coercive  on $U_{0}$. Let   $\tilde{G}: \overline{U_{0}} \times \overline{U_{0}} \setminus \{ (x,x): x\in \overline{U_{0}} \}  \longrightarrow \R$ be  the {\emph{Green's function}}  of the operator $\Delta +  ( a-a_{1}) $ with Dirichlet boundary conditions. The $\tilde{G}$ satisfies
\begin{align}{\label{Green dist def}}
\Delta \tilde{G}(x, \cdot) +  (  a -a_{1} )    \tilde{G}(x, \cdot)  = \delta_{x} 
\end{align}
Since  the operator  $\Delta + (  a-a_{1}) $ is   coercive on $U_{0}$, $\tilde{G}$ exists.  See Robert \cite{FR greens}. We set  $\tilde{G}_{\epsilon} (x)= \tilde{G}( x_{\epsilon}, x)$ for all $x \in \overline{U_{0}} \backslash \{ x_{\epsilon}\}$ and $\eps>0$. Then there exists $C>0$ such that 
\begin{align*}
0 < \tilde{G}_{\epsilon}(x) <  \frac{C}{\left|x- x_{\epsilon} \right|^{n-2}} \qquad \text{ for } x \in \overline{U_{0}} \backslash \{ x_{\epsilon}\}.
\end{align*}
Moreover there exists  $\delta_{0} >0$ and $C_{0} >0$ such that  for all $\epsilon >0$ 
\begin{align}{\label{Green estimate 2}}
\tilde{G}_{\epsilon}(x) \geq \frac{C_{0}}{\left|x- x_{\epsilon} \right|^{n-2}} ~ \text{ and } ~  \frac{| \nabla \tilde{G}_{\epsilon}(x)|}{| \tilde{G}_{\epsilon}(x) |} \geq \frac{C_{0}}{\left|x- x_{\epsilon} \right|}  \text{ for } x \in {B_{x_{\epsilon}}(\delta_{0})} \backslash \{ x_{\epsilon}\} \subset \subset U_{0}
\end{align}
We  define the operator  
\begin{align*}
\mathcal{L}_{\epsilon} = \Delta + a  - \frac{ u_{\epsilon}^{2^{*}(s_{\epsilon})-2} }{|x|^{s_{\epsilon}}} 
\end{align*}

\medskip\noindent{\bf Step 1.1:} We claim that there exists $\nu_{0}\in (0,1)$  such that given  any  $\nu \in (0, \nu_{0})$ there exists  $R_{1}>0$ such that for  $R > R_{1}$ and $\epsilon >0$ sufficiently small we have 
\begin{align}\label{to apply max principle 1}
\mathcal{L}_{\epsilon}\tilde{G}_{\epsilon}^{1-\nu} >0 \qquad \text{ in } ~ {\Omega} \backslash  B_{x_{\epsilon}}(Rk_{\epsilon})
 \end{align}
We prove the claim. We  choose  $\nu_{0} \in (0,1) $  such that for any $\nu \in (0, \nu_{0})$ one has $\nu \left(a - a_{1} \right) \geq -\frac{a_{1}}{2} $ in $\Omega$. Fix  $\nu \in (0, \nu_{0})$. Using $(\ref{Green dist def})$  we  obtain for  $\epsilon >0$ sufficiently small
\begin{align*}
\frac{\mathcal{L}_{\epsilon}\tilde{G}_{\epsilon}^{1-\nu} }{\tilde{G}_{\epsilon}^{1-\nu} } =& a_{1}   +\nu (a-a_{1} ) + \nu(1-\nu) \frac{| \nabla \tilde{G}_{\epsilon}| ^{2}}{ |\tilde{G}_{\epsilon}|^{2}}  - \frac{ u_{\epsilon}^{2^{*}(s_{\epsilon})-2} }{|x|^{s_{\epsilon}}}  \qquad \text{ in } ~ \Omega \backslash \{ x_{\epsilon}\} \notag\\
&\geq \frac{a_{1}}{2} +\nu(1-\nu) \frac{| \nabla \tilde{G}_{\epsilon}| ^{2}}{ |\tilde{G}_{\epsilon}|^{2}}  - \frac{ u_{\epsilon}^{2^{*}(s_{\epsilon})-2} }{|x|^{s_{\epsilon}}}  \qquad \qquad  \qquad \text{ in } ~ \Omega \backslash \{ x_{\epsilon}\} 
\end{align*}
Let $|x-x_{\epsilon}| \geq \delta_{0}$,where $\delta_{0}$ is as in $(\ref{Green estimate 2})$,  then  from Corollary \ref{u strongly goes to 0 locally}  we have 
\begin{align*}
\lim  \limits_{\epsilon \to 0}  \frac{ u_{\epsilon}^{2^{*}(s_{\epsilon})-2} }{|x|^{s_{\epsilon}}}=0 \qquad \text{ in } C(\overline{\Omega\backslash  B_{x_{\epsilon}}(\delta_{0}) }) 
\end{align*}
Hence for $\epsilon >0$ sufficiently small we have  for $\nu \in (0, \nu_{0})$
\begin{align*}
\frac{\mathcal{L}_{\epsilon}\tilde{G}_{\epsilon}^{1-\nu} }{\tilde{G}_{\epsilon}^{1-\nu} } >0 \qquad \text{ for  } x \in  {\Omega} \backslash  B_{x_{\epsilon}}(\delta_{0})
\end{align*}

\noindent
By strong pointwise estimates,  Proposition \ref{strong conc estimate prop}   we have that, given any $\nu \in (0, \nu_{0})$, there exists $R_{1}>0$ such that for any $R > R_{1}$ 
\begin{align*}
\sup \limits_{\Omega \backslash B_{x_{\epsilon}}(R k_{\epsilon})} \left|x-x_{\epsilon} \right|^{\frac{n-2}{2}} u_{\epsilon}(x)  \leq  \left[ \frac{\nu(1-\nu)}{4} C_{0}^{2} \right]^{\frac{n-2}{4}}
\end{align*}
Here $C_{0}$ is as in $(\ref{Green estimate 2})$. And then using Lemma \ref{strong b'dds lemma two}  we  obtain  for $\epsilon>0$ small 
\begin{align*}
  \frac{ u_{\epsilon}^{2^{*}(s_{\epsilon})-2} }{|x|^{s_{\epsilon}}}
& =   \left[ u_{\epsilon}^{2^{*}(s_{\epsilon})-2-s_{\epsilon}}  \left(\frac{ u_{\epsilon}}{|x|} \right)^{s_{\epsilon}} \right]  \notag
\leq  \frac{\nu(1-\nu)}{2} \frac{C^{2}_{0}}{\left|  x- x_{\epsilon}  \right|^{2}}   
\end{align*} 
for all $x\in \Omega \backslash B_{x_{\epsilon}}(R k_{\epsilon})$. Therefore if $x \in B_{x_{\epsilon}}(\delta_{0}) \backslash B_{x_{\epsilon}}(R k_{\epsilon})$ then with $\eqref{Green estimate 2}$ we get
\begin{align*}
\frac{\mathcal{L}_{\epsilon}\tilde{G}_{\epsilon}^{1-\nu} }{\tilde{G}_{\epsilon}^{1-\nu} } \geq  \frac{a_{1}}{2} +\frac{ \nu(1-\nu)}{2} \frac{C^{2}_{0}}{\left|  x- x_{\epsilon}  \right|^{2}}  >0
\end{align*}
for $\epsilon >0$ small. This proves the claim and ends Step 1.1.

\medskip\noindent{\bf Step 1.2:} Let  $\nu \in (0, \nu_{0})$ and $R > R_{1}$. We claim that there exists $C(R)>0$ such that for  $\epsilon >0$ small
\begin{align}\label{to apply max principle 2}
\mathcal{L}_{\epsilon} \left( C(R) \mu_{\epsilon}^{\frac{n-2}{2}- \nu(n-2)} \tilde{G}_{\epsilon}^{1-\nu} \right) &> \mathcal{L}_{\epsilon} u_{\epsilon} \qquad{ in }~ \Omega\backslash B_{x_{\epsilon}}(Rk_{\epsilon}) \notag\\
C(R) \mu_{\epsilon}^{\frac{n-2}{2}- \nu(n-2)} \tilde{G}_{\epsilon}^{1-\nu} & >  u_{\epsilon}  \ \ \  \qquad{ on }~  \partial \left( \Omega\backslash B_{x_{\epsilon}}(Rk_{\epsilon}) \right)
\end{align}
We prove the claim. Since   $\mathcal{L}_{\epsilon} u_{\epsilon} =0 $ in $\Omega$,  so it follows  from \eqref{to apply max principle 1} that $\mathcal{L}_{\epsilon} \left( C(R) \mu_{\epsilon}^{\frac{n-2}{2}- \nu(n-2)} \tilde{G}_{\epsilon}^{1-\nu} \right) > \mathcal{L}_{\epsilon} u_{\epsilon}$  in  $\Omega\backslash B_{x_{\epsilon}}(Rk_{\epsilon})$  for $R>R_{1}$ and   $\epsilon >0$ sufficiently small. With $(\ref{Green estimate 2} )$ and \eqref{lim of scaling factor},  we obtain for $\epsilon >0$ small 
\begin{align*}
\frac{u_{\epsilon}(x)}{\mu_{\epsilon}^{\frac{n-2}{2}-\nu(n-2)} \tilde{G}_{\epsilon}^{1-\nu}(x)}& \leq  \frac{\mu_{\epsilon}^{-\frac{n-2}{2}}}{\mu_{\epsilon}^{\frac{n-2}{2}-\nu(n-2)} } \frac{ \left( R k_{\epsilon}  \right)^{(n-2)(1-\nu)}}{C_{0}^{1-\nu}}  \leq  \frac{  (2R)^{(n-2)(1-\nu)}}{C_{0}^{1-\nu}}  \end{align*}
for $ x \in \Omega \cap \partial B_{ x_{\epsilon}}(Rk_{\epsilon})$. So for $x \in  \partial \left( \Omega\backslash B_{x_{\epsilon}}(Rk_{\epsilon}) \right)$ one has for $\epsilon >0$ small 
\begin{align*}
\frac{u_{\epsilon}(x)}{\mu_{\epsilon}^{\frac{n-2}{2}-\nu(n-2)} \tilde{G}_{\epsilon}^{1-\gamma}(x)} \leq C(R) \qquad \text{ for } x \in \Omega \cap \partial B_{ x_{\epsilon}}(Rk_{\epsilon})
\end{align*}
This proves the claim and ends Step 1.2.

\medskip\noindent{\bf Step 1.3:} Let  $\nu \in (0, \nu_{0})$ and $R > R_{1}$. Since $\tilde{G}_{\epsilon}^{1-\nu}>0 $ in   $\overline{ \Omega\backslash B_{x_{\epsilon}}(Rk_{\epsilon}) }$ and $\mathcal{L}_{\epsilon}\tilde{G}_{\epsilon}^{1-\nu}>0$ in $ \Omega\backslash B_{x_{\epsilon}}(Rk_{\epsilon}) $, it follows from {\cite{BNV}}  that the operator  $\mathcal{L}_{\epsilon}$ satisfies the comparison principle. Then  from \eqref{to apply max principle 2}  we have that for $\epsilon >0$ small
\begin{align*}
u_{\epsilon}(x) \leq  C(R) \mu_{\epsilon}^{\frac{n-2}{2}- \nu(n-2)} \tilde{G}_{\epsilon}^{1-\nu}(x)  \qquad \text{ for } ~ x \in  \Omega\backslash B_{x_{\epsilon}}(Rk_{\epsilon}) 
\end{align*}
Then with $(\ref{Green estimate 2})$ we get that
\begin{align*}
\left|x- x_{\epsilon} \right|^{(n-2)(1-\nu)} u_{\epsilon}(x) \leq  C(R) \mu_{\epsilon}^{\frac{n-2}{2}- \nu(n-2)}    \qquad \text{ for } ~ x \in  \Omega\backslash B_{x_{\epsilon}}(Rk_{\epsilon}) 
\end{align*}
Taking $\alpha = (n-2)(1-\nu)$, we have for $\alpha$ close to $n-2$
\begin{align*}
\left|x- x_{\epsilon} \right|^{\alpha} \mu_{\epsilon}^{\frac{n-2}{2}-\alpha} u_{\epsilon}(x) \leq  C_{\alpha}    \qquad \text{ for } ~ x \in  \Omega\backslash B_{x_{\epsilon}}(Rk_{\epsilon}) .
\end{align*}
As easily checked, this implies \eqref{strong estimate first approx}  for all  $\alpha \in (0, n-2)$. This ends Step 1.3 and also Step 1.\end{proof}
\medskip

\noindent
Next we show that one can infact take $\alpha= n-2$ in \eqref{strong estimate first approx}. 

\medskip\noindent{\bf Step 2:} We claim that there exists  $C>0$  such that  for all $\epsilon >0$
\begin{align}\label{strong estimate 2nd approx}
\left| x-x_{\epsilon} \right|^{n-2}  u_{\epsilon}(x_{\epsilon}) ~   u_{\epsilon}(x) \leq C \qquad \text{ for all } x \in  \Omega
\end{align}
\begin{proof} The claim is equivalent to proving that for any $(y_{\epsilon})_\eps \in \Omega$, we have that
\begin{align*}
\left| y_{\epsilon}-x_{\epsilon} \right|^{n-2}  u_{\epsilon}(x_{\epsilon}) ~   u_{\epsilon}(y_{\epsilon}) = O(1) \qquad { as }~ \epsilon \to 0
\end{align*}
We have the following two cases. 

\medskip\noindent{\bf Step 2.1:} Suppose that $\left| x_{\epsilon} -y_{\epsilon}\right|=O(\mu_{\epsilon})$ as $\eps\to 0$.
By  definition $(\ref{mu epsilon})$ it follows that $\left| y_{\epsilon}-x_{\epsilon} \right|^{n-2}  u_{\epsilon}(x_{\epsilon}) ~   u_{\epsilon}(y_{\epsilon})  \leq \left| y_{\epsilon}-x_{\epsilon} \right|^{n-2} \mu_{\epsilon}^{2-n}$. This proves  \eqref{strong estimate 2nd approx} in this case and ends Step 2.1.

\medskip\noindent{\bf Step 2.2:} Suppose that
\begin{align}\label{id:2}
\lim \limits_{\epsilon \to 0} \frac{\left| x_{\epsilon} -y_{\epsilon}\right|}{\mu_{\epsilon}}=+ \infty \qquad { as }~ \epsilon \to 0
\end{align}
We let for $\epsilon >0$
\begin{align*}
\hat{v}_{\epsilon}(x)= \mu_{\epsilon}^{\frac{n-2}{2}} u_{\epsilon} \left(  \mu_{\epsilon} x + x_{\epsilon} \right) \qquad \text{ for } ~ x \in \frac{\Omega - x_{\epsilon}}{\mu_{\epsilon}}
\end{align*}
Then from \eqref{strong estimate first approx}, it follows that for any $\alpha \in  (0, n-2)$, there exists  $C'_{\alpha}>0$  such that  for all $\epsilon >0$
\begin{align*}
\hat{v}_{\epsilon}(x) \leq \frac{C'_{\alpha}}{1+\left|  x \right|^{\alpha}} \qquad \text{ for } ~ x \in \frac{\Omega - x_{\epsilon}}{\mu_{\epsilon}}
\end{align*}
Let  $G$ be  the Green's function  of $\Delta +    a $ with Dirichlet boundary conditions. Green's representation formula and  standard estimates on the  Green's function yield
\begin{align*}
u_{\epsilon}(y_{\epsilon}) = \int \limits_{\Omega} G(x,y_{\epsilon}) \frac{ u_{\epsilon}^{2^{*}(s_{\epsilon})-1}(x) }{|x|^{s_{\epsilon}}} ~ dx  \leq  C\int \limits_{\Omega} \frac{1}{\left| x -y_{\epsilon}\right|^{n-2}} \frac{ u_{\epsilon}^{2^{*}(s_{\epsilon})-1}(x) }{|x|^{s_{\epsilon}}} ~ dx\hbox{ for all }\eps>0
\end{align*}
where $C>0$ is  a constant. We  write  the above integral  as follows
\begin{align*}
u_{\epsilon}(y_{\epsilon}) \leq   C\int \limits_{\Omega}  \left( \frac{u_{\epsilon}(x)}{|x|} \right)^{s_{\epsilon}} \frac{1}{\left| x -y_{\epsilon}\right|^{n-2}} ~ u_{\epsilon}(x)^{2^{*}(s_{\epsilon})-1-s_{\epsilon}} ~ dx \qquad \text{ for all } ~  \epsilon >0 
\end{align*}
Using  H\"{o}lder inequality and then by Hardy inequality $(\ref{Hardy inequality})$ we get that  for $\epsilon >0$
\begin{align*}
u_{\epsilon}(y_{\epsilon}) \leq & C  \left( ~\int \limits_{\Omega}  \frac{\left|u_{\epsilon}(x) \right|^{2}}{|x|^{2}} ~ dx \right)^{s_{\epsilon}/2} \left( \int \limits_{\Omega}  \left( \frac{1}{\left| x -y_{\epsilon}\right|^{n-2}}\right)^{\frac{2}{2-s_{\epsilon}}} ~ u_{\epsilon}(x)^{ \left( 2^{*}(s_{\epsilon})-1-s_{\epsilon} \right) \frac{2}{2-s_{\epsilon}}} ~ dx\right)^{\frac{2-s_{\epsilon}}{2}}  \notag\\
\leq &  C  \left( \left(  \frac{2}{n-2} \right)^{2}  \int \limits_{\Omega} |\nabla u_{\epsilon}|^{2} ~ dx \right)^{s_{\epsilon}/2} \left( \int \limits_{\Omega}  \left( \frac{1}{\left| x -y_{\epsilon}\right|^{n-2}}\right)^{\frac{2}{2-s_{\epsilon}}} ~ u_{\epsilon}(x)^{ \left( 2^{*}(s_{\epsilon})-1-s_{\epsilon} \right) \frac{2}{2-s_{\epsilon}}} ~ dx \right)^{\frac{2-s_{\epsilon}}{2}}   \notag\\
\end{align*} 
Since $(u_{\epsilon})_{\epsilon >0}$ is bounded in $H^{2}_{1,0}(\Omega)$,there exists $C >0$ such  that  for $\epsilon>0$ small 
\begin{align*}
u_{\epsilon}(y_{\epsilon})^{\frac{2}{2-s_{\epsilon}}} \leq  C   \int \limits_{\Omega}   \frac{1}{\left| x -y_{\epsilon}\right|^{\frac{2(n-2)}{2-s_{\epsilon}}}} ~ u_{\epsilon}(x)^{ \left( 2^{*}(s_{\epsilon})-1-s_{\epsilon} \right) \frac{2}{2-s_{\epsilon}}} ~ dx
\end{align*}
With a change of variables the above integral becomes 
\begin{align*}
u_{\epsilon}(y_{\epsilon})^{\frac{2}{2-s_{\epsilon}}}   &\leq  
  C \frac{\mu^{n}_{\epsilon}}{\mu_{\epsilon}^{\frac{n-2}{2-s_{\epsilon}}{\left( 2^{*}(s_{\epsilon})-1-s_{\epsilon} \right)  }}} \int \limits_{\frac{\Omega - x_{\epsilon}}{\mu_{\epsilon}}}   \frac{1}{\left| y_{\epsilon} -x_{\epsilon}-\mu_{\epsilon} x \right|^{\frac{2(n-2)}{2-s_{\epsilon}}}}   {\hat{v}}_{\epsilon}(x)^{ \left( 2^{*}(s_{\epsilon})-1-s_{\epsilon} \right) \frac{2}{2-s_{\epsilon}}} 
  ~ dx 
 \end{align*}
 And so we get  that  for $\epsilon>0$ small 
\begin{align}{\label{breakup of integral estimate}}
\left(\mu_{\epsilon}^{-\frac{n-2}{2}} u_{\epsilon}(y_{\epsilon}) \right)^{\frac{2}{2-s_{\epsilon}}} &\leq  C \int \limits_{ \frac{\Omega - x_{\epsilon}}{\mu_{\epsilon}} \cap \left\{ \left| y_{\epsilon} -x_{\epsilon} -\mu_{\epsilon}x \right| \geq \frac{\left| y_{\epsilon}- x_{\epsilon}\right|}{2} \right\} }\frac{1}{\left| y_{\epsilon} -x_{\epsilon}-\mu_{\epsilon} x \right|^{\frac{2(n-2)}{2-s_{\epsilon}}}}   {\hat{v}}_{\epsilon}(x)^{ \left( 2^{*}(s_{\epsilon})-1-s_{\epsilon} \right) \frac{2}{2-s_{\epsilon}}} 
  ~ dx~    \notag\\
& +  C \int \limits_{ \frac{\Omega - x_{\epsilon}}{\mu_{\epsilon}} \cap \left\{ \left| y_{\epsilon} -x_{\epsilon} -\mu_{\epsilon}x \right| \leq \frac{\left| y_{\epsilon}- x_{\epsilon}\right|}{2} \right\} } \frac{1}{\left| y_{\epsilon} -x_{\epsilon}-\mu_{\epsilon} x \right|^{\frac{2(n-2)}{2-s_{\epsilon}}}}   {\hat{v}}_{\epsilon}(x)^{ \left( 2^{*}(s_{\epsilon})-1-s_{\epsilon} \right) \frac{2}{2-s_{\epsilon}}} 
  ~ dx
\end{align}
We estimate the above two  integrals  separately. First we have for $\epsilon >0 $ small and  $\alpha$ close to $n-2$ 
\begin{eqnarray}\label{breakup of integral estimate 1}
&&\int \limits_{ \frac{\Omega - x_{\epsilon}}{\mu_{\epsilon}} \cap \left\{ \left| y_{\epsilon} -x_{\epsilon} -\mu_{\epsilon}x \right| \geq \frac{\left| y_{\epsilon}- x_{\epsilon}\right|}{2} \right\} }   \frac{1}{\left| y_{\epsilon} -x_{\epsilon}-\mu_{\epsilon} x \right|^{\frac{2(n-2)}{2-s_{\epsilon}}}}   {\hat{v}}_{\epsilon}(x)^{ \left( 2^{*}(s_{\epsilon})-1-s_{\epsilon} \right) \frac{2}{2-s_{\epsilon}}} \,
  dx \\
&&  \leq  \frac{2^{\frac{2(n-2)}{2-s_{\epsilon}}}}{\left| y_{\epsilon} - x_{\epsilon}  \right|^{\frac{2(n-2)}{2-s_{\epsilon}}}} 
\int \limits_{ \frac{\Omega - x_{\epsilon}}{\mu_{\epsilon}}}    {\hat{v}}_{\epsilon}(x)^{ \left( 2^{*}(s_{\epsilon})-1-s_{\epsilon} \right) \frac{2}{2-s_{\epsilon}}} ~ dx  \nonumber=  O \left(~ \frac{1}{\left| y_{\epsilon} - x_{\epsilon}  \right|^{\frac{2(n-2)}{2-s_{\epsilon}}}}   \right)  \end{eqnarray}
as $\eps\to 0$. On the other hand    for $\epsilon >0 $ small
\begin{align*}
&\int \limits_{ \frac{\Omega - x_{\epsilon}}{\mu_{\epsilon}} \cap \left\{ \left| y_{\epsilon} -x_{\epsilon} -\mu_{\epsilon}x \right| \leq \frac{\left| y_{\epsilon}- x_{\epsilon}\right|}{2} \right\} } \frac{1}{\left| y_{\epsilon} -x_{\epsilon}-\mu_{\epsilon} x \right|^{\frac{2(n-2)}{2-s_{\epsilon}}}}   {\hat{v}}_{\epsilon}(x)^{ \left( 2^{*}(s_{\epsilon})-1-s_{\epsilon} \right) \frac{2}{2-s_{\epsilon}}} 
  ~ dx    \notag\\
& \leq C_{\alpha} \int \limits_{ \frac{\Omega - x_{\epsilon}}{\mu_{\epsilon}} \cap \left\{ \left| y_{\epsilon} -x_{\epsilon} -\mu_{\epsilon}x \right| \leq \frac{\left| y_{\epsilon}- x_{\epsilon}\right|}{2} \right\} } \frac{1}{\left| y_{\epsilon} - x_{\epsilon} -\mu_{\epsilon} x \right|^{\frac{2(n-2)}{2-s_{\epsilon}}}} \frac{1}{|x|^{  { \left( 2^{*}(s_{\epsilon})-1-s_{\epsilon} \right) \frac{2\alpha}{2-s_{\epsilon}}}} } ~ dx     \notag\\
& \leq C_{\alpha} \left( \frac{2 \mu_{\epsilon}}{\left| y_{\epsilon} - x_{\epsilon}  \right|} \right)^{ { \left( 2^{*}(s_{\epsilon})-1-s_{\epsilon} \right) \frac{2\alpha}{2-s_{\epsilon}}}} \int \limits_{  \left\{ \left| y_{\epsilon} -x_{\epsilon} -\mu_{\epsilon}x \right| \leq \frac{\left| y_{\epsilon}- x_{\epsilon}\right|}{2} \right\} } \frac{1}{\left| y_{\epsilon} -x_{\epsilon}-\mu_{\epsilon} x \right|^{\frac{2(n-2)}{2-s_{\epsilon}}}}    ~ dx   \notag\\
& \leq C_{\alpha} \left( \frac{ \mu_{\epsilon}}{\left| y_{\epsilon} - x_{\epsilon}  \right|} \right)^{ { \left( 2^{*}(s_{\epsilon})-1-s_{\epsilon} \right) \frac{2\alpha}{2-s_{\epsilon}}}-n}  \left(  \frac{1}{\left| y_{\epsilon} - x_{\epsilon}  \right|^{n-2} }  \right)^{ \frac{2}{2-s_{\epsilon}}}
\end{align*}
Taking $\alpha$ close to $(n-2)$, and using \eqref{id:2}, we obtain for $\epsilon$ sufficiently small 
\begin{equation}{\label{breakup of integral estimate 2}}
\int \limits_{ \frac{\Omega - x_{\epsilon}}{\mu_{\epsilon}} \cap \left\{ \left| y_{\epsilon} -x_{\epsilon} -\mu_{\epsilon}x \right| \leq \frac{\left| y_{\epsilon}- x_{\epsilon}\right|}{2} \right\} } \frac{ {\hat{v}}_{\epsilon}^{2^{*}(s_{\epsilon})-1}(x) }{\left| y_{\epsilon} - x_{\epsilon} -\mu_{\epsilon} x \right|^{n-2}} ~ dx   =  o \left(  \frac{1}{\left| y_{\epsilon} - x_{\epsilon}  \right|^{n-2} }  \right)^{ \frac{2}{2-s_{\epsilon}}} 
\end{equation}
as $\eps\to 0$. Combining $(\ref{breakup of integral estimate})$, $(\ref{breakup of integral estimate 1})$ and $(\ref{breakup of integral estimate 2})$ we obtain that 
\begin{align*}
\left(\mu_{\epsilon}^{-\frac{n-2}{2}} u_{\epsilon}(y_{\epsilon}) \right)^{\frac{2}{2-s_{\epsilon}}}   \leq O\left(~ \frac{1}{\left| y_{\epsilon} - x_{\epsilon}  \right|^{\frac{2(n-2)}{2-s_{\epsilon}}}}   \right)  \qquad { as }~ \epsilon \to 0
\end{align*} 
This proves  \eqref{strong estimate 2nd approx} and ends Step 2.2 and then Step 2.\end{proof}

\medskip\noindent{\bf Step 3:} The estimate \eqref{strong estimate 2nd approx} and the definition \eqref{mu epsilon} of $\mu_{\epsilon}$ yield  Theorem \ref{th:1}. \end{proof}

\section{Localizing the Singularity: The Interior Blow-up Case}\label{sec:loc:int}

\noindent
In this section we prove  Theorem \ref{th:2}. We assume that
$$x_0\in\Omega.$$
The proof goes through four steps. We first recall the Pohozaev identity. Let $U$ be a bounded smooth domain in $\R^{n}$,  let $p_{0} \in \R^{n}$ be  a point and let  $u \in C^{2}(\overline{U})$. We have
\begin{align}{\label{pohozaev identity}} 
\int \limits_{U}  \left( (x-p_{0}, \nabla u)  + \frac{n-2}{2} u  \right) \Delta u~ dx = \int \limits_{\partial U} \left( (x-p_{0}, \nu) \frac{|\nabla u|^{2}}{2}- \left(   (x-p_{0}, \nabla u)   + \frac{n-2}{2}u  \right) \partial_{\nu} u\right)  d\sigma
\end{align}
here $\nu$ is the outer normal to the boundary $\partial U$. Using the above Pohozaev Identity  we obtain the following  identity for the Hardy Sobolev equation: Let $U_{\epsilon}$ be  a family of smooth domains such that $ x_{\epsilon} \in U_{\epsilon} \subset \Omega$  for all  $\epsilon >0$.  One has for all  $\epsilon >0$
\begin{align}{\label{pohozaev identity for HS}} 
&\int \limits_{U_{\epsilon} }  \left( a  +  \frac{(x-x_{\epsilon}, \nabla a ) }{2}\right) u_{\epsilon}^{2} ~ dx - \frac{s_{\epsilon}(n-2)}{2(n-s_{\epsilon})}  \int \limits_{U_{\epsilon} } \frac{u_{\epsilon}^{2^{*}(s_{\epsilon}) }}{|x|^{s_{\epsilon}}} \frac{(x, x_{\epsilon})}{|x|^{2}} ~ dx=  \notag\\
&\int \limits_{\partial U_{\epsilon} }  (x-x_{\epsilon}, \nu) \left(\frac{|\nabla u_{\epsilon}|^{2}}{2} + \frac{a u_{\epsilon}^{2}}{2}-  \frac{1}{2^{*}(s_{\epsilon})}  \frac{u_{\epsilon}^{2^{*}(s_{\epsilon}) }}{|x|^{s_{\epsilon}}} \right)  d\sigma  -  \int \limits_{\partial U_{\epsilon} }     
\left(   (x-x_{\epsilon}, \nabla u_{\epsilon})   + \frac{n-2}{2} u_{\epsilon}  \right) \partial_{\nu} u_{\epsilon}   ~d\sigma
\end{align}

\medskip\noindent
Since $x_{0} \in \Omega$, let $\delta >0$ be such that $B_{x_{0}}(3\delta) \subset \Omega$. Note that then $\lim \limits_{\epsilon \to 0} |x_{\epsilon}|^{s_{\epsilon}}=1$,  and  it follows from \eqref{lim of scaling factor} that $\lim \limits_{\epsilon \to 0} \mu_{\epsilon}^{s_{\epsilon}}=1$. We will estimate each of the terms in the above Pohozaev identity and calculate the limit as $\epsilon \to $ and $\delta \to 0$. It will depend on the dimension $n$. 

\medskip\noindent{\bf Step 1:} We prove the following convergence outside $x_0$:
\begin{proposition}{\label{convergence to green's function 1}} We have that $\mu_{\epsilon}^{-\frac{n-2}{2}} u_{\epsilon} \longrightarrow b_{n} G_{x_{0}}$ in $C^{1}_{loc} ( \overline{\Omega }\setminus \{ x_{0} \})$ as $\epsilon \to 0$, where $b_n$ is as in \eqref{def:dn:bn} and $G$ is the Green's function for $\Delta+a$ with Dirichlet condition.
\end{proposition}
\begin{proof} We fix $y_{0} \in \Omega$ such that $y_{0} \neq x_{0}$.  We first claim that
\begin{align*}
\lim \limits_{\epsilon \to 0 } ~ \mu_{\epsilon}^{-\frac{n-2}{2}} u_{\epsilon}(y_{0}) \longrightarrow b_{n} G_{x_{0}}(y_{0}). 
\end{align*}
\noindent We prove the claim. We choose $\delta' \in (0, \delta) $ such that  $|x_{0}-y_{0}| \geq 3 \delta'$ and $|x_{0}| \geq 3 \delta'$. From Green's representation formula we have 
\begin{align*}
\mu_{\epsilon}^{-\frac{n-2}{2}}  u_{\epsilon}(y_{0})& =   \mu_{\epsilon}^{-\frac{n-2}{2}} \int \limits_{B_{x_{\epsilon}}(\delta')} G(x,y_{0}) \frac{ u_{\epsilon}^{2^{*}(s_{\epsilon})-1}(x) }{|x|^{s_{\epsilon}}} ~  dx + \mu_{\epsilon}^{-\frac{n-2}{2}} \int \limits_{\Omega \setminus B_{x_{\epsilon}}(\delta')} G(x,y_{0}) \frac{ u_{\epsilon}^{2^{*}(s_{\epsilon})-1}(x) }{|x|^{s_{\epsilon}}}~   dx 
\end{align*}
Using the bounds on $u_{\epsilon}$ obtained in  Theorem \ref{th:1}  and  the  estimates on the Green's function $G$ we get  as $\epsilon \to 0$
\begin{align*}
\mu_{\epsilon}^{-\frac{n-2}{2}} u_{\epsilon}(y_{0}) =    \mu_{\epsilon}^{-\frac{n-2}{2}} \int \limits_{B_{x_{\epsilon}}(\delta')} G(x,y_{0}) \frac{ u_{\epsilon}^{2^{*}(s_{\epsilon})-1}(x) }{|x|^{s_{\epsilon}}} ~  dx   +  O(\mu_{\epsilon}^{2-s_{\epsilon}} )  .
\end{align*}
Recall our definition of $v_{\epsilon}$ in Theorem \ref{first blowup theorem}. With a  change of variable, Theorem \ref{first blowup theorem} yields
\begin{align*}
\mu_{\epsilon}^{-\frac{n-2}{2}} u_{\epsilon}(y_{0}) = \left( \frac{ |x_{\epsilon} |^{s_{\epsilon}}}{\mu_{\epsilon}^{s_{\epsilon}}} \right)^{ \frac{n-2}{2}} \int \limits_{B_{0}(\delta' k^{-1}_{\epsilon} )} G(x_{\epsilon} + k_{\epsilon} x,y_{0}) \frac{ v_{\epsilon}^{2^{*}(s_{\epsilon})-1}(x) }{\left|   \frac{x_{\epsilon} }{|x_{\epsilon}|}+ \frac{k_{\epsilon}}{|x_{\epsilon}|} x\right|^{s_{\epsilon}}} ~  dx +  O(\mu_{\epsilon}^{2-s_{\epsilon}} ) 
\end{align*}
Lebesgue dominated convergence theorem,  Theorems \ref{first blowup theorem}  and  \ref{th:1} then yield 
\begin{align}\label{lim:out:0}
\lim \limits_{ \epsilon \to 0} \mu_{\epsilon}^{-\frac{n-2}{2}} u_{\epsilon}(y_{0}) = G(x_{0}, y_{0}) \int \limits_{\R^{n}} v^{2^{*}-1}  dx = b_nG(x_{0}, y_{0}).
\end{align}
This proves the claim. From \eqref{the eqn}, we get that
\begin{align*}
\Delta ( \mu_{\epsilon}^{-\frac{n-2}{2}} u_{\epsilon} ) + a (x)  (\mu_{\epsilon}^{-\frac{n-2}{2}} u_{\epsilon}) =& \mu_{\epsilon}^{2-s_{\epsilon}} \frac{ (\mu_{\epsilon}^{-\frac{n-2}{2}} u_{\epsilon})^{2^{*}(s_{\epsilon})-1} }{\left| x\right|^{s_{\epsilon}}} \qquad   \text{ in }  \Omega \notag\\
\mu_{\epsilon}^{-\frac{n-2}{2}} u_{\epsilon}=&0 \qquad \text{ on }~ \partial \Omega. 
\end{align*}
It follows from the pointwise estimate of Theorem \ref{th:1} that  $\mu_{\epsilon}^{-\frac{n-2}{2}} u_{\epsilon}$ is uniformly bounded in $L^\infty_{loc}(\Omega\setminus\{x_0\})$. It then follows from standard elliptic theory that the limit \eqref{lim:out:0} holds in $C^1_{loc}(\overline{\Omega}\setminus\{x_0\})$. This completes the proof  of Proposition \ref{convergence to green's function 1}.\end{proof}

\medskip\noindent{\bf Step 2:}
Next we show  that 
\begin{align}\label{convergence to v_{L^{2*}}}
 \lim \limits_{\epsilon \to0 }\int \limits_{B_{x_{\epsilon}}(\delta) } \frac{ u_{\epsilon}^{2^{*}(s_{\epsilon}) }}{|x|^{s_{\epsilon}}} \frac{(x, x_{\epsilon})}{|x|^{2}} ~ dx =  \left( \frac{1}{K(n,0)} \right)^{\frac{2^{*}}{2^{*}-2}}.
\end{align}
\begin{proof} Recall our definition of $v_{\epsilon}$ in   Theorem \ref{first blowup theorem}.  With a change of variable  we have 
\begin{align*}
 \int \limits_{B_{x_{\epsilon}}(\delta) } \frac{ u_{\epsilon}^{2^{*}(s_{\epsilon}) }}{|x|^{s_{\epsilon}}} \frac{(x, x_{\epsilon})}{|x|^{2}} ~ dx  =   \left(  \frac{ \left|x_{\epsilon}\right|^{s_{\epsilon}}}{\mu_{\epsilon}^{s_{\epsilon}}} \right) ^{\frac{n-2}{2}}
  \int \limits_{B_{0} (\delta / k_{\epsilon})}  \frac{(x_{\epsilon}+ k_{\epsilon}x, x_{\epsilon})}{|x_{\epsilon}+ k_{\epsilon}x |^{2}}  \frac{v_{\epsilon}(x)^{2^{*}(s_{\epsilon})}}{ \left| \frac{x_{\epsilon}}{|x_{\epsilon}|} + \frac{k_{\epsilon}}{|x_{\epsilon}|}x \right|^{s_{\epsilon}}} ~ dx 
\end{align*}
Passing to limits,   and using  Theorems \ref{first blowup theorem}  and \ref{th:1} we obtain by Lebesgue dominated convergence theorem
\begin{align*}
\lim \limits_{\epsilon \to 0} \int \limits_{B_{x_{\epsilon}} ( \delta)} \frac{u_{\epsilon}^{2^{*}(s_{\epsilon}) }}{|x|^{s_{\epsilon}}} \frac{(x, x_{\epsilon})}{|x|^{2}} ~ dx =  \int  \limits_{\R^{n}} v^{2^{*}} ~ dx =  \left( \frac{1}{K(n,0)} \right)^{\frac{2^{*}}{2^{*}-2}}.
\end{align*}
This proves \eqref{convergence to v_{L^{2*}}} and ends Step 2.\end{proof}

\medskip\noindent{\bf Step 3:} We define $a_\eps(x):=a(x)  +  \frac{1}{2}(x-x_{\epsilon}, \nabla a ) $ for $x\in\Omega$. We claim that
\begin{align}\label{est:l2}
\int_{B_{x_\eps}(\delta)}  a_\eps u_{\epsilon}^{2} ~ dx  =   \left \{ \begin{array} {lc}
          O(\delta \mu_{\epsilon})   &\text{for } n=3\hbox{ or }a\equiv 0,\\
          \mu^{2}_{\epsilon} \log \left( \frac{1}{k_{\epsilon}}\right) \left[64 \omega_{3} a(x_{0})+ o(1) \right] &\text{for } n=4,\\
           \mu_{\epsilon}^2 \left[ d_{n} a(x_{0})+ o(1)\right]      &\text{for } n\geq 5  .
            \end{array} \right. 
\end{align}
as $\epsilon \to 0$, where $d_n$ is as in \eqref{def:dn:bn}.
\begin{proof} We divide the proof in three steps.\par
\noindent {\bf Case 3.1:} We assume that $n\geq 5$. Recall our definition of $v_{\epsilon}$ in  Theorem \ref{first blowup theorem}. With a  change of variable we obtain  
\begin{align*}
  \mu_{\epsilon}^{-2}  \int \limits_{B_{x_{\epsilon}}(\delta) } a_\eps u_{\epsilon}^{2} ~ dx = 
  \left( \frac{k_{\epsilon}}{\mu_{\epsilon}}  \right)^{4}\int \limits_{B_{0}( \delta / k_{\epsilon}) } a_\eps(x_{\epsilon} + k_{\epsilon} x) v_{\epsilon}^{2} ~ dx. 
\end{align*}
Theorem \ref{th:1} reads $v_\eps(x)\leq C(1+|x|^2)^{1-n/2}$. Therefore, Lebesgue's theorem and Theorem \ref{first blowup theorem} yield \eqref{est:l2} when $n\geq 5$.

\noindent {\bf Case 3.2:} We assume that $n=4$ and we argue as in Case 3.1. With the pointwise control of Theorem \ref{th:1}, we get that
$$\int_{B_{0}( \delta / k_{\epsilon}) } a_\eps(x_{\epsilon} + k_{\epsilon} x) v_{\epsilon}^{2} ~ dx. =\log \left( \delta/ k_{\epsilon}\right)\left(64 \omega_{3}~ a(x_{0})+o(1)\right)\hbox{ as }\eps\to 0.$$

\smallskip\noindent{\bf Case 3.3:} we assume that $n=3$. It follows from Theorem \ref{th:1} that there exists $C>0$ such that $ \mu_{\epsilon}^{-1/2} u_{\epsilon}(x) \leq C|x -x_{\epsilon}|^{-1}$ for all $\epsilon >0$ and $x\in\Omega$. Therefore 
\begin{align*}
& \int_{B_{x_\eps}(\delta)} a_\eps u_{\epsilon}^{2} ~ dx =  O(\mu_{\epsilon})  \int_{B_{x_\eps}(\delta) }  |x|^{-2}\, dx =  O(\delta\mu_{\epsilon})  \hbox{ as }\eps\to 0.
\end{align*}

\end{proof}
\noindent{\bf Step 3:} We prove  Theorem \ref{th:2} for $n\geq 4$. From the Pohozaev identity  \eqref{pohozaev identity for HS} we have  
\begin{align}{\label{pohozaev identity for HS, n>4}}
& \mu_{\epsilon}^{-2} \int \limits_{B_{x_{\epsilon}}(\delta) }  \left( a  +  \frac{(x-x_{\epsilon}, \nabla a ) }{2}\right) u_{\epsilon}^{2} ~ dx - \mu_{\epsilon}^{-2}\frac{s_{\epsilon}(n-2)}{2(n-s_{\epsilon})}  \int \limits_{B_{x_{\epsilon}}(\delta) } \frac{ u_{\epsilon}^{2^{*}(s_{\epsilon}) }}{|x|^{s_{\epsilon}}} \frac{(x, x_{\epsilon})}{|x|^{2}} ~ dx  \notag\\
=& \mu_{\epsilon}^{n-4} \int \limits_{\partial B_{x_{\epsilon}}(\delta) }  (x-x_{\epsilon}, \nu) \left(\frac{|\nabla (  \mu_{\epsilon}^{-\frac{n-2}{2}} u_{\epsilon}  )|^{2}}{2} + \frac{a}{2}  ( \mu_{\epsilon}^{-\frac{n-2}{2}}  u_{\epsilon})^{2}-  \frac{\mu_{\epsilon}^{2-s_{\epsilon}}}{2^{*}(s_{\epsilon})}  \frac{( \mu_{\epsilon}^{-\frac{n-2}{2}}  u_{\epsilon})^{2^{*}(s_{\epsilon}) }}{|x|^{s_{\epsilon}}} \right)  d\sigma  \notag\\
& -  \mu_{\epsilon}^{n-4}\int \limits_{\partial B_{x_{\epsilon}}(\delta) }     
\left(   (x-x_{\epsilon}, \nabla ( \mu_{\epsilon}^{-\frac{n-2}{2}}  u_{\epsilon}))   + \frac{n-2}{2}( \mu_{\epsilon}^{-\frac{n-2}{2}}  u_{\epsilon}) \right) \partial_{\nu} ( \mu_{\epsilon}^{-\frac{n-2}{2}}  u_{\epsilon})   ~d\sigma
\end{align} 
Passing to the limits as $\epsilon \to 0$ in \eqref{pohozaev identity for HS, n>4}, using \eqref{convergence to v_{L^{2*}}}, \eqref{est:l2} and Theorem \ref{convergence to green's function 1}, we get Theorem \ref{th:2} when $n\geq 4$.

\medskip\noindent{\bf Step 4:} We now deal with the case of dimension $n=3$. Recall from the introduction that we write the Green's function $G$ as $G_x(y)=\frac{1}{4\pi|x-y|}+g_x(y)$ for all $x,y\in\Omega$, $x\neq y$, with $g_{x} \in C^{2}( \overline{\Omega } \setminus \{x\})\cap C^{0, \theta}(\Omega)$ for some $0 < \theta <1$. In particular, when $n=3$ or $a\equiv 0$, $g_x(x)$ is defined for all $x\in \Omega$. For any $x \in \Omega$,  $g_{x}$ satifies the equation
\begin{align*}
\Delta g_{x} + a g_{x}= -a/(4\pi|x-y|)  \hbox{ in }~ \Omega \setminus\{ x\} \hbox{ and } g_{x} (y)= \frac{-1}{\omega_{2}|x-y|}\hbox{ on } \partial \Omega .
\end{align*}
\noindent Note that  any  $x \in {\Omega}$ 
\begin{align}\label{green's function n=3 estimates}
\lim \limits_{r \to 0}~ \sup \limits_{ y \in \partial B_{x}(r)} ~ |y-x| |\nabla g_{x}(y)| =0 
\end{align}
The proof goes as in Hebey-Robert \cite{hebeyrobertACV}. We omit it here. From  the Pohozaev identity \eqref{pohozaev identity for HS}, multiplying both the sides by $\mu_{\epsilon}^{-1}$ we obtain  

\begin{align}{\label{pohozaev identity for HS, n=3}}
&\int \limits_{B_{x_{\epsilon}}(\delta) }  \left( a  +  \frac{(x-x_{\epsilon}, \nabla a ) }{2}\right) (\mu_{\epsilon}^{-1/2} u_{\epsilon})^{2} ~ dx - \frac{s_{\epsilon}}{2\mu_{\epsilon}(3-s_{\epsilon})}  \int \limits_{B_{x_{\epsilon}}(\delta) } \frac{u_{\epsilon}^{2^{*}(s_{\epsilon}) }}{|x|^{s_{\epsilon}}} \frac{(x, x_{\epsilon})}{|x|^{2}} ~ dx=  \notag\\
& \int \limits_{\partial B_{x_{\epsilon}}(\delta) }  (x-x_{\epsilon}, \nu) \left(\frac{|\nabla ( \mu_{\epsilon}^{-1/2} u_{\epsilon}) |^{2}}{2} + a \frac{  ( \mu_{\epsilon}^{-1/2} u_{\epsilon})^{2}}{2}- \frac{\mu_{\epsilon}^{2-s_{\epsilon}}}{2^{*}(s_{\epsilon})} \frac{ ( \mu_{\epsilon}^{-1/2} u_{\epsilon})^{2^{*}(s_{\epsilon}) }}{|x|^{s_{\epsilon}}} \right)  d\sigma \notag\\
 & -  \int \limits_{\partial B_{x_{\epsilon}}(\delta) }  \left(   (x-x_{\epsilon}, \nabla ( \mu_{\epsilon}^{-1/2} u_{\epsilon}))   + \frac{n-2}{2}  ( \mu_{\epsilon}^{-1/2} u_{\epsilon}) \right) \partial_{\nu}  ( \mu_{\epsilon}^{-1/2} u_{\epsilon})   ~d\sigma
\end{align} 
It follows from Proposition \ref{convergence to green's function 1} that 
\begin{align*}
&\lim \limits_{\epsilon \to 0} \int \limits_{\partial B_{x_{\epsilon}}(\delta) }  (x-x_{\epsilon}, \nu) \left(\frac{|\nabla ( \mu_{\epsilon}^{-1/2} u_{\epsilon}) |^{2}}{2} + a \frac{  ( \mu_{\epsilon}^{-1/2} u_{\epsilon})^{2}}{2}- \frac{\mu_{\epsilon}^{2-s_{\epsilon}}}{2^{*}(s_{\epsilon})} \frac{ ( \mu_{\epsilon}^{-1/2} u_{\epsilon})^{2^{*}(s_{\epsilon}) }}{|x|^{s_{\epsilon}}} \right)  d\sigma \notag\\
&  - \lim \limits_{\epsilon \to 0}  \int \limits_{\partial B_{x_{\epsilon}}(\delta) }  \left(   (x-x_{\epsilon}, \nabla ( \mu_{\epsilon}^{-1/2} u_{\epsilon}))   + \frac{n-2}{2}  ( \mu_{\epsilon}^{-1/2} u_{\epsilon}) \right) \partial_{\nu}  ( \mu_{\epsilon}^{-1/2} u_{\epsilon})   ~d\sigma \notag\\
& = b^{2}_{3} \int \limits_{\partial B_{x_{0}} (\delta) } \delta    \frac{|\nabla G_{x_{0}} |^{2}}{2} + \frac{ \delta a }{2}  (G_{x_{0}})^{2}     - \frac{ (x-x_{0}, \nabla G_{x_{0}} )^{2}}{\delta}   -  \frac{n-2}{2}  \frac{ (x-x_{0}, \nabla G_{x_{0}} )}{\delta}  G_{x_{0}}     ~d\sigma 
\end{align*}
Using  \eqref{green's function n=3 estimates}, we get that the right-hand-side goes to $ \frac{ b_3^2  }{ 2 } g_{x_{0}}(x_{0})$ as $\delta\to 0$. Putting this identity, \eqref{est:l2} when $n=3$, and \eqref{convergence to v_{L^{2*}}} in \eqref{pohozaev identity for HS, n=3} , we get Theorem \ref{th:2} in the case $n=3$. The proof is similar when $a\equiv 0$.

\section{Localizing the Singularity: The Boundary Blow-up Case}\label{sec:loc:bdy}

\noindent
This section is devoted to the proof of Theorem \ref{th:3}.

\subsection{Convergence to Singular Harmonic Functions}$~$
Here, $G$ is still the Green's function of the coercive operator $\Delta + a $ in $\Omega$ with Dirichlet boundary conditions. The following result for the asymptotic analysis of the Green's function is in the spirit of Proposition 5  of  \cite{FRgreens} and Proposition 12 of \cite{DRW}.
\begin{theorem}[\cites{FRgreens,DRW}]{\label{blow up for  greens}}
Let $(x_{\epsilon})_{\epsilon>0} \in \Omega$ and let  $(r_{\epsilon})_{\epsilon>0} \in (0, + \infty)$ be such that $\lim \limits_{\epsilon \to 0} r_{\epsilon}=0$. 
\begin{enumerate}
\item[(1)] Assume that $\lim_{\epsilon \to 0} \frac{d(x_{\epsilon}, \partial \Omega)}{r_{\epsilon}} = + \infty$. Then  for all $x, y \in \R^{n}$, $x \neq y$, we have that 
\begin{align*}
\lim \limits_{\epsilon \to 0}~ r_{\epsilon}^{n-2} G(x_{\epsilon}+ r_{\epsilon}x , x_{\epsilon}+ r_{\epsilon}y )= \frac{1}{(n-2) \omega_{n-1} |x-y|^{n-2}}
\end{align*}
where $\omega_{n-1}$ is the area of the $(n-1)$- sphere. Moreover for a fixed $x \in \R^{n}$, this convergence holds uniformly in $C^{2}_{loc}(\R^{n} \backslash \{ x\})$. \\
\item[(2)] Assume that $\lim_{\epsilon \to 0} \frac{d(x_{\epsilon}, \partial \Omega)}{r_{\epsilon}} = \rho \in [0, + \infty)$. Then $\lim \limits_{\epsilon \to 0} x_{\epsilon}= x_{0} \in \partial \Omega$. Let $\mathcal{T}$ be a parametrisation  of the boundary $\partial \Omega$ as in \eqref{def:T:bdy} around the point $p=x_{0}$.
We write  $ \mathcal{T}^{-1}(x_{\epsilon})= ((x_{\epsilon})_{1}, x'_{\epsilon})$. Then  for all $x, y \in \R^{n} \cap \{ x_{1} \leq 0\}$, $x \neq y$, we have that 
\begin{align*} 
&\lim \limits_{\epsilon \to 0} ~r_{\epsilon}^{n-2} G \left(  \mathcal{T}(  (0, x'_{\epsilon})+ r_{\epsilon}x) , \mathcal{T}(  (0, x'_{\epsilon})+ r_{\epsilon}y ) \right)   \notag\\
&= \frac{1}{(n-2) \omega_{n-1} |x-y|^{n-2}}- \frac{1}{(n-2) \omega_{n-1} | \pi(x)-y|^{n-2}}
\end{align*}
where $ \pi : \R^{n} \to \R^{n}$ defined by $\pi((x_{1}, x')) \mapsto (-x_{1}, x')$ is the reflection across the plane $\{ x: x_{1}=0\}$. 
Moreover for a fixed $x \in  \overline{\R^{n}_{-}}$, this convergence holds uniformly in $C^{2}_{loc}(  \overline{\R^{n}_{-}} \backslash \{ x\})$. \\
\end{enumerate} 
\end{theorem}

The next proposition shows that   the pointwise behaviour of the blowup sequence $(u_{\epsilon})_{\epsilon >0}$ is well approximated  by bubbles. Note that the following proposition holds with $x_0\in \overline{\Omega}$, in the interior or on the boundary. We omit the proof as it goes exactly like   the  proof of  Proposition 13 in  \cite{DRW} .
\begin{proposition}{\label{conv to bubble}}
We set for all  $ \epsilon > 0$
\begin{align*}
U_{\epsilon}(x)=  \left( \frac{ k_{\epsilon}}{ k_{\epsilon}^{2} +\frac{|x-x_{\epsilon}|^{2}}{n(n-2)}  }\right)^{\frac{n-2}{2}} 
\end{align*}
Suppose that the sequence  $\left(u_{\epsilon} \right)_{\epsilon>0} \in H^{2}_{1,0}(\Omega)$, where for  each  $\epsilon >0$,  $u_{\epsilon}$  satisfies  $(\ref{the eqn})$ and $(\ref{min energy condition})$, is a  {\emph{blowup sequence}}. We let $x_0:=\lim_{\epsilon\to 0}x_\epsilon$. Let $(y_{\epsilon})_{\epsilon >0}$ be a sequence of points in $ \overline{\Omega}$. We have
\begin{enumerate}
\item[(1)] If $\lim_{\epsilon\to 0}y_\epsilon=y_0\neq x_0$, then $\lim \limits_{ \epsilon \to 0} \mu_{\epsilon}^{-\frac{n-2}{2}} u_{\epsilon}(y_{\epsilon}) = b_n G_{x_0}(y_{0}) $ (see Proposition \ref{convergence to green's function 1}).
\item[(2)] If $\lim_{\epsilon\to 0}y_\epsilon=x_0$ and $\lim \limits_{\epsilon \to 0} d(x_{\epsilon}, \partial \Omega)>0$, then 
\begin{align*}
u(y_{\epsilon}) = (1+ o(1)) U_{\epsilon}(y_{\epsilon})  \qquad ~ \text{ as } \epsilon \to 0 
\end{align*}
\item[(3)]  If $\lim_{\epsilon\to 0}y_\epsilon=x_0$ and $\lim \limits_{\epsilon \to 0} d(x_{\epsilon}, \partial \Omega)=0$, then 
\begin{align*}
u(y_{\epsilon}) = (1+ o(1)) \left( U_{\epsilon}(y_{\epsilon}) - {\tilde{U}}_{\epsilon}(y_{\epsilon}) \right)  \qquad  ~ \text{ as } \epsilon \to 0 
\end{align*}
where for $ \epsilon > 0$
\begin{align*}
\tilde{U}_{\epsilon}(x)=   \left( \frac{k_{\epsilon}}{ k_{\epsilon}^{2} + \frac{\left|x- \pi_{\mathcal{T}}(x_{\epsilon} )\right|^{2} }{n(n-2)}}\right)^{\frac{n-2}{2}} 
\end{align*}
where $\pi_{\mathcal{T}}= \mathcal{T} \circ \pi \circ \mathcal{T}^{-1} $. Here, $\mathcal{T}$ and $\pi$  are as in Theorem \ref{blow up for greens}.\end{enumerate}
\end{proposition}

\medskip \noindent
Using  Proposition \ref{conv to bubble}, we derive the following  when the sequence of blowup points converge to a point on the  boundary
 \begin{proposition}{\label{convergence to green's function 2}} Let $\left(u_{\epsilon} \right)_{\epsilon>0} \in H^{2}_{1,0}(\Omega)$ be such that for  each  $\epsilon >0$,  $u_{\epsilon}$  satisfies  $(\ref{the eqn})$ and $(\ref{min energy condition})$. We assume that $u_\epsilon\rightharpoonup 0$ weakly in $H^{2}_{1,0}(\Omega)$ as $\epsilon\to 0$. We let $x_0:=\lim_{\epsilon\to 0}x_\epsilon$. Let  $r _{\epsilon}= d(x_{\epsilon}, {\partial \Omega})$. We assume that $\lim \limits_{\epsilon \to 0} r_{\epsilon}=0$. Therefore, $\lim \limits_{\epsilon \to 0} x_{\epsilon} = x_{0}\in\partial\Omega$.  Let $\mathcal{T}$ be a parametrisation  of the boundary $\partial \Omega$ as in \eqref{def:T:bdy} around the point $p=x_{0}$. We write  $ \mathcal{T}^{-1}(x_{\epsilon})= ((x_{\epsilon})_{1}, x'_{\epsilon})$. For $\epsilon >0$,  let
\begin{align*}
 \tilde{v}_{\epsilon}(x) := \frac{r_{\epsilon}^{n-2}}{\mu_{\epsilon}^{\frac{n-2}{2}}} u_{\epsilon} \circ  \mathcal{T}(  (0, x'_{\epsilon})+ r_{\epsilon}x)   \qquad \text { for } ~x \in \frac{U- (0, x'_{\epsilon}) }{r_{\epsilon}} \cap \{ x_{1} \leq 0\}
 \end{align*}
 Then
 \begin{align*}
 \lim \limits_{\epsilon \to 0}  \tilde{v}_{\epsilon}(x)  =  (n(n-2))^{\frac{n-2}{2}} \left(  \frac{1}{ |x-\theta_{0} |^{n-2}}- \frac{1}{ | x-\pi(\theta_{0})|^{n-2}}  \right)  \hbox{ in }~ C^{1}_{loc}( \overline{\R^{n}_{-}}\setminus \{ \theta_{0} \} )
 \end{align*}
where 
\begin{align}\label{def:theta}
\theta_{0} = \lim \limits_{\epsilon \to 0} \theta_{\epsilon} , \qquad  \theta_{\epsilon}= \left( \frac{(x_{\epsilon})_{1}}{r_{\epsilon}}, 0 \right)\in \R^{n}_{-}
\end{align}
and $ \pi : \R^{n} \to \R^{n}$ defined by $\pi((x_{1}, x')) \mapsto (-x_{1}, x')$ is the reflection across the plane $\{ x: x_{1}=0\}$. 
 \end{proposition}
 \begin{proof}
Since $D_{0} \mathcal{T} = \mathbb{I}_{\R^{n}}$  we have:  $ d(x_{\epsilon}, \partial \Omega)= \left( 1+ o(1) \right) |(x_{\epsilon})_{1}|$.  Let $\theta_{\epsilon}$ be as in \eqref{def:theta}. Then we have that $\theta_{0}= \lim \limits_{\epsilon \to 0} \theta_{\epsilon} = (-1,0) \in \R^{n}_{-}$ and $\pi (\theta_{0}) = (1,0) \in \R^{n}_{+}$. We fix $R>0$.  $\tilde{v}_{\epsilon}$   is defined in  $B_{0}(R) \cap \{ x_{1} \leq 0\} $ for $\epsilon >0$ small. It follows from  the strong upper bounds obtained in Theorem \ref{th:1} that  there exists a  constant $C>0$ such that for  $\epsilon >0$  small we have
 \begin{align*}
 0 \leq \tilde{v}_{\epsilon}(x) \leq C   \left(    \frac{ r^{2}_{\epsilon}}{ |  \mathcal{T}(  (0, x'_{\epsilon})+ r_{\epsilon}x) -x_{\epsilon}|^{2} } \right)^{\frac{n-2}{2}} \qquad \text{ for } x \in B_{0}(R) \cap \{ x_{1} < 0 \}
 \end{align*}
For any    $x \in B_{0}(R) \cap \{ x_{1} \leq 0\} $ we get from Proposition \ref{conv to bubble}   that as $\epsilon \to 0$
\begin{align}
 \tilde{v}_{\epsilon}(x)
= (1+ o(1))  \left( \frac{k_{\epsilon}}{\mu_{\epsilon}} \right)^{\frac{n-2}{2} }    \left(  \left( \frac{ 1}{\left( \frac{k_{\epsilon}}{r_{\epsilon}} \right)^{2} +\frac{|  \mathcal{T}(  (0, x'_{\epsilon})+ r_{\epsilon}x) -x_{\epsilon}|^{2}}{n(n-2) r^{2}_{\epsilon}}  }\right)^{\frac{n-2}{2}}  -   \left( \frac{1 }{ \left( \frac{k_{\epsilon}}{r_{\epsilon}} \right)^{2}+ \frac{\left| \mathcal{T}(  (0, x'_{\epsilon})+ r_{\epsilon}x) - \pi_{\mathcal{T}}^{-1}(x_{\epsilon} )\right|^{2} }{n(n-2) r^{2}_{\epsilon}} }\right)^{\frac{n-2}{2}}   \right) \label{eq:ve}
\end{align}
Fom the properties of the boundary map $\mathcal{T}$, one then gets 
\begin{align}
\lim \limits_{\epsilon \to 0} \tilde{v}_{\epsilon}(x)=   \frac{(n(n-2))^{\frac{n-2}{2}} }{ |x-(1,0) |^{n-2}}- \frac{(n(n-2))^{\frac{n-2}{2}} }{ | x+(1,0) |^{n-2}}  \text{ for } x \in \left( B_{0}(R)  \setminus \{ (1,0)\} \right)\cap \{ x_{1} \leq 0\}\label{lim:ve}
\end{align}
For $i,j=1, \ldots,n$, we let $({\tilde{g}}_{\epsilon})_{ij}(x)= \left( \partial_{i} \mathcal{T}\left( (0, x_{\epsilon}') +r_{\epsilon}x \right), \partial_{j} \mathcal{T}\left( (0, x_{\epsilon}') +r_{\epsilon}x \right) \right)$,  the induced metric on the domain $B_{0}(R) \cap \{ x_{1} <0\}$, and let $\Delta_{g}$ denote the Laplace-Beltrami operator with respect to the metric $g$. From eqn $(\ref{the eqn})$ it follows that  given  any $R>0$,   $  \tilde{v}_{\epsilon}$  weakly satisfies  the  following equation   for  $\epsilon >0$ sufficiently small 
\begin{eqnarray}
 \left \{ \begin{array} {lc}
          \Delta_{\tilde{g}_{\epsilon}}  \tilde{v}_{\epsilon}  + r^{2}_{\epsilon} \left( a \circ   \mathcal{T}(  (0, x'_{\epsilon})+ r_{\epsilon}x)   \right)    \tilde{v}_{\epsilon} = \left( \frac{\mu_{\epsilon}}{r_{\epsilon}}  \right)^{2-s_{\epsilon}}\frac{   \tilde{v}_{\epsilon}^{2^{*}(s_{\epsilon})-1} }{\left| \frac{  \mathcal{T}(  (0, x'_{\epsilon})+ r_{\epsilon}x)  }{ r_{\epsilon} }\right|^{s_{\epsilon}}} &  \text{in }  B_{0}(R) \cap  \{ x_{1 } <0 \} \\  \\
            \tilde{v}_{\epsilon}=0  & \text{on }  B_{0}(R) \cap  \{ x_{1 } =0 \}
            \end{array} \right. 
\end{eqnarray}
Arguing as in Step 1.2 of the proof of Lemma \ref{blowup lemma}, we get that the convergence of $\tilde{v}_{\epsilon}  $ holds in $C^{1}_{loc}( \overline{\R^{n}_{-}}\setminus \{ \theta_{0} \} )$. This completes the proof  of Proposition \ref{convergence to green's function 2}.
 \end{proof}

 \subsection{Estimates on the blow up rates: The Boundary Case} $~$
Suppose that the sequence  of blow up points $(x_{\epsilon})_{\epsilon>0}$ converges to a point on  the boundary, i.e suppose $  \lim \limits_{\epsilon \to 0 } x_{\epsilon} = x_{0} \in \partial \Omega$.  We let  
 \begin{align}\label{eq:12}
 r_{\epsilon}= d(x_{\epsilon}, \partial \Omega)
 \end{align}
 Then $ \lim \limits_{ \epsilon \to 0} r_{\epsilon} =0$ and from \eqref{second lim}, we have as $\epsilon \to 0 $: $\mu_{\epsilon} = o(r_{\epsilon})$ and $ k_{\epsilon} = o(r_{\epsilon})$. We apply the Pohozaev identity for the Hardy Sobolev equation \eqref{pohozaev identity for HS}  to the domain $B_{x_{\epsilon}}(r_{\epsilon}/2)$. Note that since $\frac{d(x_{\epsilon}, \partial \Omega)}{r_{\epsilon}}=1$ for all $\epsilon >0$, so    $\overline{B_{x_{\epsilon}}(r_{\epsilon}/2)}  \subset\subset \Omega$  for $\epsilon >0$ small. The Pohozaev identity \eqref{pohozaev identity for HS} gives us 
\begin{align}{\label{pohozaev identity for HS 2}} 
&\int_{B_{x_{\epsilon}}(r_{\epsilon}/2)}  \left( a  +  \frac{(x-x_{\epsilon}, \nabla a ) }{2}\right) u_{\epsilon}^{2} ~ dx - \frac{s_{\epsilon}(n-2)}{2(n-s_{\epsilon})}  \int_{B_{x_{\epsilon}}(r_{\epsilon}/2)}  \frac{u_{\epsilon}^{2^{*}(s_{\epsilon}) }}{|x|^{s_{\epsilon}}} \frac{(x, x_{\epsilon})}{|x|^{2}} ~ dx =\notag\\
&\int_{\partial B_{x_{\epsilon}}(r_{\epsilon}/2)}  (x-x_{\epsilon}, \nu) \left(\frac{|\nabla u_{\epsilon}|^{2}}{2} + \frac{a u_{\epsilon}^{2}}{2}-  \frac{1}{2^{*}(s_{\epsilon})}  \frac{u_{\epsilon}^{2^{*}(s_{\epsilon}) }}{|x|^{s_{\epsilon}}} \right) - \left(   (x-x_{\epsilon}, \nabla u_{\epsilon})   + \frac{n-2}{2} u_{\epsilon}  \right) \partial_{\nu} u_{\epsilon} ~d\sigma  
\end{align}
for all  $\epsilon >0$ small. We now estimate each of the terms in the integral above. Theorem \ref{th:3} will be a consequence of the following theorem:
\begin{theorem}\label{th:bdy:bis} Let $\Omega$, $a$, $(s_{\epsilon})_{\epsilon >0}$,  $\left(u_{\epsilon} \right)_{\epsilon>0} \in H^{2}_{1,0}(\Omega)$ as in Theorem \ref{th:3}. 
Assume that \eqref{eq:12} holds and $  \lim \limits_{\epsilon \to 0 } x_{\epsilon} = x_{0} \in \partial \Omega$. Then 
\begin{enumerate}
\item[(1)] If $n=3$ or $a\equiv 0$, then  as $\epsilon \to 0$
\begin{align}\label{res:1}
 \lim_{\epsilon\to 0} \frac{s_{\epsilon}r_{\epsilon}^{n-2}}{\mu_{\epsilon}^{n-2}}=  \frac{n^{n-1}(n-2)^{n-1} K(n,0)^{{n/2}} \omega_{n-1}  }{2^{n-2}}.
\end{align}
Moreover, $d(x_\epsilon,\partial\Omega)=(1+o(1))|x_\epsilon|$ as $\epsilon\to 0$. In particular $x_0=0$.
\\
\item[(2)]  If $n=4$. Then  as $\epsilon \to 0$
\begin{align}\label{res:2}
  \frac{s_{\epsilon}}{4} \left(  K(4,0)^{-2} + o(1) \right) - \left( \frac{\mu_{\epsilon}}{r_{\epsilon}}\right)^{2} \left(
  32  \omega_{3} + o(1)  \right)=   \mu^{2}_{\epsilon} \log \left( \frac{r_{\epsilon}}{\mu_{\epsilon}}\right) \left[d_{4} a(x_{0})+ o(1) \right] 
\end{align}
and
\begin{equation*}
s_{\epsilon }  \left(1-\left(\frac{r_\epsilon}{|x_\epsilon|}\right)^2+o(1)\right)=\mu^{2}_{\epsilon} \log \left( \frac{r_{\epsilon}}{\mu_{\epsilon}}\right) \left[4d_{4} K(4,0)^2a(x_{0})+ o(1) \right] 
\end{equation*}\\
\item[(3)] If $n \geq 5$. Then  as $\epsilon \to 0$
\begin{align}\label{res:3}
   \frac{s_{\epsilon}(n-2)}{2n} \left( K(n,0)^{-n/2}+ o(1) \right) -\left( \frac{\mu_{\epsilon}}{r_{\epsilon}}\right)^{n-2} \left(
   \frac{ n^{n-2}(n-2)^{n} \omega_{n-1} }{ 2^{n-1}}  + o(1)  \right) = \mu_{\epsilon}^2 \left[ d_{n} a(x_{0})+ o(1)\right]
\end{align}
and
\begin{equation*}
s_{\epsilon}   \left(1-\left(\frac{r_\epsilon}{|x_\epsilon|}\right)^2+o(1)\right)=\mu^{2}_{\epsilon} \left[\frac{2n}{n-2}d_{n} K(n,0)^2a(x_{0})+ o(1) \right] 
\end{equation*}
\end{enumerate}
where $d_n$ is as in \eqref{def:dn:bn}  for  $n\geq 5 $ and $d_{4} =  64\omega_3$.
\end{theorem}
\begin{proof}
For convenience we define
\begin{align*}
 F_{\epsilon}=(x-x_{\epsilon}, \nu) \left(\frac{|\nabla u_{\epsilon}|^{2}}{2} + \frac{a u_{\epsilon}^{2}}{2}-  \frac{1}{2^{*}(s_{\epsilon})}  \frac{u_{\epsilon}^{2^{*}(s_{\epsilon}) }}{|x|^{s_{\epsilon}}} \right) - \left(   (x-x_{\epsilon}, \nabla u_{\epsilon})   + \frac{n-2}{2} u_{\epsilon}  \right) \partial_{\nu} u_{\epsilon}
\end{align*}

\medskip\noindent{\bf Step 1:} We claim that
\begin{align}\label{id:step1}
 \left( \frac{\mu_{\epsilon}}{r_{\epsilon}}\right)^{2-n} \int_{\partial B_{x_{\epsilon}}(r_{\epsilon}/2)}  F_{\epsilon} ~ d \sigma  = 
 -  \frac{ n^{n-2}(n-2)^{n} \omega_{n-1} }{ 2^{n-1}}  + o(1) \qquad \text{ as } ~ \epsilon \to 0
\end{align}
\begin{proof} We define 
\begin{equation}\label{def:ve}
\hat{v}_{\epsilon}(x) := \frac{r_{\epsilon}^{n-2}}{\mu_{\epsilon}^{\frac{n-2}{2}}} u_{\epsilon}(x_\epsilon+ r_{\epsilon}x) \text { for } ~x \in B_0(1).
\end{equation}
Since $d(x_\epsilon,\partial\Omega)=r_\eps$, this is well-defined. Moreover, Proposition \ref{convergence to green's function 2} yields
\begin{align}\label{lim:ve:sing}
 \lim_{\epsilon \to 0}  \hat{v}_{\epsilon} =   \hat{v}(x)=\frac{(n(n-2))^{\frac{n-2}{2}} }{ |x|^{n-2}}- \frac{(n(n-2))^{\frac{n-2}{2}} }{ | x-(2,0) |^{n-2}} \hbox{ in } C^{1}_{loc}( B_0(1) ).
 \end{align}
With the change of variable $ x \mapsto x_\epsilon + r_{\epsilon } z $ we obtain 
 \begin{eqnarray*}
&&  \left( \frac{\mu_{\epsilon}}{r_{\epsilon}}\right)^{2-n}  \int_{\partial B_{x_{\epsilon}}(r_{\epsilon}/2)} F_{\epsilon} ~ d \sigma 
 =   \int \limits_{\partial  B_0 (1/2) }   \left(  z, \nu \right) \frac{|\nabla \hat{v}_{\epsilon}|^{2}}{2} ~ d \sigma \\
&&+  \int \limits_{\partial  B_0 (1/2) }   \left( z, \nu \right)   r_{\epsilon}^{2} a \left( x_{\epsilon} + r_{\epsilon } z \right)\frac{ \hat{v}_{\epsilon}^{2}}{2}  ~ d \sigma -   \int \limits_{\partial  B_0 (1/2) }   \frac{\left(z, \nu \right)}{2^{*}(s_{\epsilon})} \left( \frac{\mu_{\epsilon}}{r_{\epsilon}}\right)^{2-s_{\epsilon}} \frac{\hat{v}_{\epsilon}^{2^{*}(s_{\epsilon}) }}{\left| \frac{x_{\epsilon} + r_{\epsilon } z  }{r_{\epsilon}} \right|^{s_{\epsilon}}}  ~ d \sigma   \\
&&- \int \limits_{\partial  B_0 (1/2) }  \left(\left( z,  \nabla \hat{v}_{\epsilon} \right)  + \frac{n-2}{2} {\hat{v}}_{\epsilon}  \right)   \partial_{\nu} \hat{v}_{\epsilon}  ~ d \sigma \label{int:Fe}
\end{eqnarray*}
Passing to limit as $\epsilon \to 0$ in \eqref{int:Fe} and using \eqref{lim:ve:sing}, we get 
\begin{equation}\label{lim:1}
  \left( \frac{\mu_{\epsilon}}{r_{\epsilon}}\right)^{2-n} \int_{\partial B_{x_{\epsilon}}(r_{\epsilon}/2)}F_{\epsilon} ~ d \sigma  = A(1/2)+o(1)\hbox{ as }\eps\to 0
 \end{equation}
 where $$A(\delta):=\int _{\partial  B_{0 } (\delta) } \left( \left(z, \nu \right) \frac{|\nabla \hat{v}|^{2}}{2} - \left( \left(z, \nabla \hat{v} \right)+ \frac{n-2}{2}v\right) \partial_{\nu} \hat{v}\right)  ~ d\sigma.$$
Let $0< \delta < 1/2 $. Since $\Delta \hat{v}=0$ in $B_{0}(1/2) \setminus B_{0}(\delta) $, applying the Pohozaev identity \eqref{pohozaev identity}, we see that $A(\delta)=A(1/2)$ for all $0<\delta<1/2$. We write 
 \begin{align}\label{def:h}
  \hat{v}(x)  =  \frac{(n(n-2))^{\frac{n-2}{2}} }{ |x |^{n-2}}+ h(x) \qquad \text{ for } x \in B_0(1)\setminus \{ 0 \} 
 \end{align}
where  $h(x)= - \frac{(n(n-2))^{\frac{n-2}{2}} }{ | x+(2,0) |^{n-2}} $. With the explicit expression of $v$ we obtain   
\begin{align*}
\lim \limits_{\delta \to 0}  A(\delta)= -  \frac{ n^{n-2}(n-2)^{n} \omega_{n-1} }{ 2^{n-1}} 
\end{align*}
Since $A$ is constant, this latest limit and \eqref{lim:1} yield \eqref{id:step1}. This completes Step 1. 
\end{proof}

\medskip\noindent{\bf Step 2:}  Proceeding similarly as in \eqref{convergence to v_{L^{2*}}} we  obtain 
\begin{align*}
 \int_{B_{x_\eps}(r_\eps/2)} \frac{u_{\epsilon}^{2^{*}(s_{\epsilon}) }}{|x|^{s_{\epsilon}}} \frac{(x, x_{\epsilon})}{|x|^{2}} ~ dx =   \left( \frac{1}{K(n,0)} \right)^{\frac{2^{*}}{2^{*}-2}} + o(1) \qquad \text{ as }~ \epsilon \to 0
\end{align*}

\medskip\noindent{\bf Step 3:} Arguing as in the proof of \eqref{est:l2}, we get that
\begin{align*}
\int_{B_{x_\eps}(r_\eps/2)}  \left( a  +  \frac{(x-x_{\epsilon}, \nabla a ) }{2}\right) u_{\epsilon}^{2} ~ dx  =   \left \{ \begin{array} {lc}
          O(\mu_{\epsilon})   &\text{for } n=3\hbox{ or }a\equiv 0,\\
          \mu^{2}_{\epsilon} \log \left( \frac{r_{\epsilon}}{k_{\epsilon}}\right) \left[64 \omega_{3} a(x_{0})+ o(1) \right] &\text{for } n=4,\\
           \mu_{\epsilon}^2 \left[ d_{n} a(x_{0})+ o(1)\right]      &\text{for } n\geq 5  .
            \end{array} \right. 
\end{align*}
where $d_n$ is as in \eqref{def:dn:bn}.
\medskip\noindent Combining Steps 1 to 3 in the Pohozaev identity \eqref{pohozaev identity for HS 2} yields \eqref{res:1}, \eqref{res:2} and \eqref{res:3}.

\medskip\noindent To get extra informations, we differentiate the Pohozaev identity \eqref{pohozaev identity for HS} with respect to the $j^{th}$ variable $(x_{\epsilon})_j$ and get 
\begin{align}{\label{pohozaev identity for HS 3}} 
&\int_{B_{x_\eps}(r_\eps/2)} \frac{\partial_{j} a}{2} u_{\epsilon}^{2}\,dx + \frac{s_{\epsilon}(n-2)}{2(n-s_{\epsilon})}  \int_{B_{x_\eps}(r_\eps/2)} \frac{u_{\epsilon}^{2^{*}(s_{\epsilon}) }}{|x|^{s_{\epsilon}}} \frac{x_{j}}{|x|^{2}} ~ dx =\notag\\
&\int_{\partial B_{x_\eps}(r_\eps/2)} \left( \nu_{j} \left(\frac{|\nabla u_{\epsilon}|^{2}}{2} + \frac{a u_{\epsilon}^{2}}{2}-  \frac{1}{2^{*}(s_{\epsilon})}  \frac{u_{\epsilon}^{2^{*}(s_{\epsilon}) }}{|x|^{s_{\epsilon}}} \right) - \partial_{j} u_{\epsilon} \partial_{\nu} u_{\epsilon} \right) ~d\sigma
\end{align}

\medskip\noindent{\bf Step 4:} We claim that
\begin{align}
 \frac{\mu_{\epsilon}^{2-n}}{r_{\epsilon}^{1-n}} \int_{\partial B_{x_\eps}(r_\eps/2)} \left( \nu_{1} \left(\frac{|\nabla u_{\epsilon}|^{2}}{2} + \frac{a u_{\epsilon}^{2}}{2}-  \frac{1}{2^{*}(s_{\epsilon})}  \frac{u_{\epsilon}^{2^{*}(s_{\epsilon}) }}{|x|^{s_{\epsilon}}} \right) - \partial_{1} u_{\epsilon} \partial_{\nu} u_{\epsilon}  \right)~d\sigma      \notag\\
=  -  \frac{  n^{n-2}(n-2)^{n} \omega_{n-1}}{ 2^{n-1} }   + o(1)
\end{align}
\begin{proof} As Step 1 above, using Proposition \ref{convergence to green's function 2} we have  as $\epsilon \to 0$
 \begin{align}\label{derived pohozaev}
&   \frac{\mu_{\epsilon}^{2-n}}{r_{\epsilon}^{1-n}}  \int_{\partial B_{x_\eps}(r_\eps/2)} \left( \nu_{j} \left(\frac{|\nabla u_{\epsilon}|^{2}}{2} + \frac{a u_{\epsilon}^{2}}{2}-  \frac{1}{2^{*}(s_{\epsilon})}  \frac{u_{\epsilon}^{2^{*}(s_{\epsilon}) }}{|x|^{s_{\epsilon}}} \right) - \partial_{j} u_{\epsilon} \partial_{\nu} u_{\epsilon}  \right)~d\sigma    \notag\\
&= \int \limits_{\partial B_{0} (1/2) } \left(  \nu_{j}  \frac{|\nabla \hat{v}|^{2}}{2} - \partial_{j} \hat{v} ~\partial_{\nu} \hat{v} \right)  ~ d\sigma +o(1)  \notag
\end{align}
where $\hat{v}_\epsilon$ and $\hat{v}$ are as in Step 1 above. Arguing as in Step 1 above , we get that
\begin{align}
 \int \limits_{\partial  B_{0 } (1/2) } \left(  \nu_{j}  \frac{|\nabla \hat{v}|^{2}}{2} - \partial_{j} \hat{v} ~\partial_{\nu} \hat{v} \right)  ~ d\sigma=    \omega_{n-1} (n-2) (n(n-2))^{\frac{n-2}{2}}\partial_{j} h(0),
\end{align}
where $h$ is as in \eqref{def:h}. For $j=1$,  taking the explicit expression of $h$ yields Step 4. \end{proof}

\medskip\noindent{\bf Step 5:} Arguing as in Step 2 we have 
\begin{align}\label{eq:30}
 \int_{B_{x_\eps}(r_\eps/2)}\frac{u_{\epsilon}^{2^{*}(s_{\epsilon}) }}{|x|^{s_{\epsilon}}} \frac{x_{1}}{|x|^{2}} ~ dx =  \frac{(x_{\epsilon})_{1} }{|x_{\epsilon}|^{2} } \left( \frac{1}{K(n,0)} \right)^{\frac{2^{*}}{2^{*}-2}}    \left(1+ o(1)  \right) \qquad \text{ as }~ \epsilon \to 0
\end{align}

\medskip\noindent Similarly, as in Step 3, for every $1 \leq j \leq n$ we have as $\epsilon \to 0$
\begin{align}
  \int_{B_{x_\eps}(r_\eps/2)} \partial_{j} a(x)~ u_{\epsilon}^{2}(x)   ~ dx  =   \left \{ \begin{array} {lc}
          O(\mu_{\epsilon}) &\text{ for}~ n=3,\\
          O\left(\mu^{2}_{\epsilon} \log \left( \frac{r_{\epsilon}}{k_{\epsilon}}\right)\right)   &\text{ for  }~ n=4,\\
          O\left( \mu_{\epsilon}^2\right)  &\text{ for }~ n\geq 5  .
            \end{array} \right. 
\end{align}
Using the Pohozaev identity \eqref{pohozaev identity for HS 3}, \eqref{res:3} and these estimates, noting that $r_\epsilon=d(x_\epsilon,\partial\Omega)= (1+o(1))|x_{\epsilon,1}|$, we then obtain that $d(x_\epsilon,\partial\Omega)=(1+o(1))|x_\epsilon|$ as $\epsilon\to 0$ when $n=3$ or $a\equiv 0$. When $n=4$, then  as $\epsilon \to 0$
\begin{align*}
  \frac{s_{\epsilon}}{4}\frac{(x_{\epsilon})_{1}}{|x_{\epsilon}|^{2}} \left(  K(4,0)^{-2} + o(1) \right) + & \frac{\mu_{\epsilon}^{2}}{r_{\epsilon}^{3}} \left(
  32  \omega_{3}+ o(1)  \right)  =O\left(\mu_{\epsilon}^2 \log \left( \frac{r_{\epsilon}}{\mu_{\epsilon}}\right)  \right) .
\end{align*}
Finally, when $n\geq 5$, we get as $\epsilon \to 0$
\begin{align*}
   \frac{s_{\epsilon}(n-2)}{2n} \frac{(x_{\epsilon})_{1}}{|x_{\epsilon}|^{2}} \left( K(n,0)^{-n/2}+ o(1) \right) &+r_{\epsilon}^{-1}\left( \frac{\mu_{\epsilon}}{r_{\epsilon}}\right)^{n-2} \left(
   \frac{ n^{n-2}(n-2)^{n} \omega_{n-1}}{ 2^{n-1} }  + o(1)  \right) \notag\\
  & = O\left(\mu_{\epsilon}^2\right)
\end{align*}
Plugging together these estimates and \eqref{res:2} and \eqref{res:3}, we get Theorem \ref{th:bdy:bis}. 
\end{proof}

\end{document}